\documentclass[11pt]{article}
\usepackage{amsmath,amsfonts,amsthm,graphicx,amssymb}
\theoremstyle{theorem}
\newtheorem{thm}{Theorem}[section]
\newtheorem*{thm1}{Theorem}
\newtheorem{prop}[thm]{Proposition}
\newtheorem{lem}[thm]{Lemma}
\newtheorem{cor}[thm]{Corollary}
\newtheorem*{cor1}{Corollary}
\newtheorem*{prop1}{Proposition}

\theoremstyle{definition}

\newtheorem{defn}[thm]{Definition}

\newtheorem{exap}[thm]{Example}
\newtheorem{rem}[thm]{Remark}
\numberwithin{equation}{section}

\newtheorem{quest}{Question}

\newcommand\Sym{{\rm {Sym}}}

\newcommand\mult{\widehat{{h}}^1_x}
\newcommand\vol{{\rm {vol}}}
\newcommand\Vol{{\rm {vol}_{{BdFF}}}}
\newcommand\Env{{\rm {Env}}}
\newcommand\hl{{{h^1_{x}}}}
\newcommand\Spec{{\rm{Spec}}}
\newcommand\Hom{{\rm{Hom}}}

\input xy
\xyoption{all}
\date{}
\title{Local volumes on normal algebraic varieties}
\author{Mihai Fulger}
\begin{document}
\maketitle
\tableofcontents

\section*{Introduction}
\addcontentsline{toc}{section}{Introduction}
In this paper we study a notion of volume for Cartier divisors on arbitrary blow-ups of normal complex algebraic varieties of dimension greater than one, with a distinguished point. We apply this to study a volume for normal isolated singularities, generalizing \cite{JW1}. We also compare this volume of isolated singularities to a different generalization by \cite{BdFF}.

\par Plurigenera of smooth complex projective varieties have been the object of much research in complex birational geometry. More recent, local analogues have been studied in \cite{W}, \cite{Y}, \cite{I2} and \cite{M} as invariants of isolated singularities appearing on normal complex algebraic varieties. For a normal complex isolated algebraic singularity $(X,x)$ of dimension $n$ at least two, the plurigenera of $(X,x)$ in the sense of Morales (\cite{M}) are defined as the dimensions of skyscraper sheaves:
$$\lambda_m(X,x)=_{\rm def}\dim\frac{\mathcal O_X(mK_X)}{\pi_*\mathcal O_Y(mK_{\widetilde X}+mE)},$$ where $\pi:(\widetilde X,E)\to(X,x)$ is an arbitrary log-resolution. One sees that $\lambda_m(X,x)=0$ if $x$ is a smooth point of $X$ or, more generally, if $X$ is $\mathbb Q-$Gorenstein with log-canonical singularities. 

\par The growth rate of $\lambda_m(X,x)$ is studied in \cite{I2} and \cite{W}. It is shown that $\lambda_m(X,x)$ grows at most like $m^n$. A natural object to study is then the finite asymptotic limit
$$\vol(X,x)=_{\rm def}\limsup_{m\to\infty}\frac{\lambda_m(X,x)}{m^n/n!}$$ that we call \textit{the volume of the singularity} $(X,x)$. For surfaces, $\vol(X,x)$ has been studied in \cite{JW1} and shown to be a characteristic number of the link of the singularity. In particular, its behavior under pull-back by ramified maps was analyzed. The vanishing of $\vol(X,x)$ in the two dimensional case is also well understood.  We will see that many of its other properties generalize to higher dimension.   

\vskip .3cm
\par We introduce a local invariant that includes the volume of isolated singularities as a special case. Let $X$ be a normal algebraic variety of dimension at least two over an algebraically closed field of arbitrary characteristic and let $x$ be a point on $X$. Fixing a proper birational map $\pi:X'\to X$, 
for an arbitrary Cartier divisor $D$ on $X'$, define the \textit{local volume of $D$ at $x$} to be 
$$\vol_x(D)=_{\rm def}\limsup_{m\to\infty}\frac{\hl(mD)}{m^n/n!},$$ where
$$\hl(D)=_{\rm def}\dim H^1_{\{x\}}(X,\pi_*\mathcal O_{X'}(D)).$$

We show that $\vol_x(D)$ is finite. When $\pi:(\widetilde X,E)\to(X,x)$ is a log-resolution of a normal complex isolated singularity of dimension $n$, we will see that $$\vol(X,x)=\vol_x(K_{\widetilde X}+E).$$

Drawing parallels between the theory of local volumes and the theory of asymptotic cohomological functions on projective varieties as presented in \cite{K} or \cite[Ch.2.2.C]{L}, we prove:

\begin{thm1} Let $\pi:X'\to X$ be a proper birational map and let $x$ be a point on the normal algebraic variety of $X$ or dimension $n$ at least two. Then $\vol_x$ is well defined, $n-$homogeneous and continuous on $N^1(X'/X)_{\mathbb R}$.\end{thm1}

\noindent As usual, $N^1(X'/X)_{\mathbb R}$ denotes the additive group of $\mathbb R-$Cartier divisors on $X'$ modulo numerical equivalence on the fibers of $\pi$. A difference between $\vol_x$ and the volume of divisors on projective varieties is that whereas the latter increases in all effective directions, $\vol_x$ decreases in effective directions that contract to $x$ and increases in effective directions without components contracting to $x$. This behavior proves  quite useful. Following ideas in \cite{LM}, we present a convex geometry approach to local volumes that allows us to prove the following:
\begin{prop1}Let $\pi:X'\to X$ be a proper birational map and let $x$ be a point on the normal algebraic variety of $X$ or dimension $n$ at least two. For any Cartier divisor $D$ on $X'$, we can replace $\limsup$ in the definition of $\vol_x(D)$ by $\lim$:
$$\vol_x(D)=\lim_{m\to\infty}\frac{\hl(mD)}{m^n/n!}.$$\end{prop1}

In the style of \cite[Thm.3.8]{LM}, we obtain a Fujita approximation type result. If $\mathcal I$ is a fractional ideal sheaf on $X$, we define its local multiplicity at $x$ to be:
$$\mult(\mathcal I)=_{\rm def}\limsup_{m\to\infty}\frac{\dim H^1_{\{x\}}(\mathcal I^m)}{m^n/n!}.$$

\begin{thm1} Let $\pi:X'\to X$ be a proper birational morphism with $X$ normal algebraic of dimension at least two and let $x$ be a point on $X$. On $X'$, let $D$ be a Cartier divisor such that the graded family $\mathfrak a_p=\pi_*\mathcal O_{X'}(pD)$ is divisorial outside $x$. Then
$$\vol_x(D)=\lim_{p\to\infty}\frac{\mult(\pi_*\mathcal O_{X'}(pD))}{p^n}.$$\end{thm1}

Two other problems that are well understood in the projective case are the vanishing and log-concavity for volumes of Cartier divisors (see \cite[Ch.2.2.C]{L}). We know that volumes vanish outside the big cone and that $\vol^{1/n}$ is a concave function on the same big cone. In the local setting we find analogous results when working with divisors supported on the fiber over $x$. Denote by ${\rm Exc}_x(\pi)$ the real vector space spanned by all such divisors.
\begin{prop1}On $X'$, let $D$ be a Cartier divisor supported on the fiber over $x$. Then $\vol_x(D)=0$ if, and only if, $D$ is an effective divisor. When $D$ is an arbitrary Cartier divisor, then $\vol_x(D)=0$ if, and only if, $\hl(m\widetilde D)=0$ for all $m
\geq 0$, where $\widetilde D$ is the pullback of $D$ to the normalization of $X'$.\end{prop1}
\begin{prop1}The function $\vol_x$ is log-convex on ${\rm Exc}_x(\pi)$, but it may fail to be so on $N^1(X'/X)_{\mathbb R}$.\end{prop1}

Returning to the setting of normal complex isolated singularities, we generalize to higher dimension some of the properties established in \cite{JW1} for local volumes of isolated surface singularities. Unlike the two dimensional case, we show in Example \ref{nottop} that $\vol(X,x)$ is not a topological invariant of the link of the singularity in dimension at least three.

\begin{prop1} Let $f:(X,x)\to (Y,y)$ be a finite map of complex normal isolated singularities of dimension $n$ with $f(x)=y$. Then
$$\vol(X,x)\geq (\deg f)\cdot\vol(Y,y).$$ Equality holds if $f$ is unramified outside $y$.
\end{prop1}

\begin{cor1}
\begin{enumerate}
\item If $f:(X,x)\to(Y,y)$ is a finite map of normal isolated singularities and $\vol(X,x)$ vanishes, then $\vol(Y,y)=0$.
\item If $(X,x)$ admits an endomorphism of degree at least two, then $\vol(X,x)=0$.
\end{enumerate}
\end{cor1}

\par For surfaces, the vanishing of $\vol(X,x)$ is equivalent to $(X,x)$ being log-canonical  in the sense of \cite[Rem.2.4]{JW1}. In arbitrary dimension, as a corollary to \cite[Thm.4.2]{I2}, we show:

\begin{prop1} If $(X,x)$ is a normal isolated singularity of dimension $n$, then $\vol(X,x)=0$ if, and only if, $\lambda_m(X,x)=0$ for all $m\geq 0$.\end{prop1} 
\noindent In the $\mathbb Q-$Gorenstein case, the conclusion of the previous result is the same as saying that $(X,x)$ has log-canonical singularities, but by \cite{BdFF} this is not the case in general. We may construct another notion of volume that is useful for the study of canonical singularities in the sense of \cite{dFH}:
$$\vol_{\gamma}(X,x)=_{\rm def}\vol_x(K_{\widetilde X}),$$ where $\pi:\widetilde X \to X$ is a resolution of a normal isolated singularity $(X,x)$. We will see that $\vol_{\gamma}(X,x)$ is also independent of the resolution.
\begin{prop1}If $(X,x)$ is a normal complex isolated singularity, then $\vol_{\gamma}(X,x)=0$ if, and only if, $(X,x)$ has canonical singularities in the sense of \cite{dFH}.\end{prop1}

\par On surfaces, we mention that by \cite{JW1}, the volume $\vol(X,x)$ can be computed as $-P\cdot P$ where $P$ is the nef part of the relative Zariski decomposition of $K_{\widetilde X}+E$  for any good resolution $\pi:(\widetilde X,E)\to (X,x)$. Building on the theory of $b-$divisors, this definition is generalized to higher dimension in \cite{BdFF} to produce another notion of volume for a normal isolated singularity, denoted $\Vol (X,x)$. We are able to show
$$\Vol(X,x)\geq\vol(X,x).$$ By the same \cite{BdFF}, the two notions of volume differ in general, but coincide in the $\mathbb Q-$Gorenstein case and we are able to slightly extend this to the numerically Gorenstein case (cf. \cite{BdFF}). The volume $\Vol(X,x)$ enjoys similar properties to those of $\vol(X,x)$ concerning the behavior with respect to finite covers and is better suited for the study of log-canonical singularities. On the other hand, $\Vol(X,x)$ is usually hard to compute because all birational models of $X$ may influence it as opposed to $\vol(X,x)$, which is computed on any log-resolution of $(X,x)$.

\par For illustration, consider the case of cone singularities. Let $(V,H)$ be a non-singular polarized complex projective variety of dimension $n$ and let  $X$ be the cone $\Spec\bigoplus_{m\geq 0}H^0(V,\mathcal O(mH))$ whose vertex $0$ is an isolated singularity. By explicit computation, or by \cite[Thm.1.7]{W},
$$\lambda_m(X,0)=\sum_{k\geq 1}\dim H^0(V,\mathcal O(mK_V-kH)).$$ We will see that this leads to, $$\vol(X,0)=(n+1)\cdot \int_0^{\infty}\vol(K_V-tH)dt.$$ The volume under the integral is the volume of line bundles on projective varieties in the sense of \cite[Ch.2.2.C]{L}. All isolated surface singularities have rational volume, but cone singularities provide examples of isolated singularities with irrational volume $\vol(X,x)$ already in dimension three. As we will see, combining techniques in \cite{BdFF} with results in our study of $\vol_x$, the volume $\Vol(X,x)$ can also be computed for some cone singularities and it can also achieve irrational values. 

\vskip .3cm

The paper is organized as follows. After the introduction and setting notation and conventions, section one develops the theory of local volumes and we compute several examples in the first subsection, before presenting a convex geometry approach to local volumes and proving our version of the Fujita approximation theorem. We next investigate the vanishing and log-convexity for $\vol_x$. Section two is dedicated to the volume of isolated singularities associated to the plurigenera in the sense of Watanabe or Morales and to $\vol_{\gamma}(X,x)$, an asymptotic invariant associated to Kn\"oller's plurigenera. We generalize to higher dimension results for surfaces in \cite{JW1}, translate to volumes some of the results of Ishii (\cite{I2}) and give examples. In section three, we compare our notion of volume with the one appearing in \cite{BdFF}. By studying the impact that the theory of $\vol_x$ has on $\Vol(X,x)$, we are able to give a non-trivial computation for $\Vol(X,x)$ that yields an irrational result. We end with a list of open questions in section four.

\vskip .3cm
\noindent\textbf{Acknowledgments.} The author would like to express his gratitude to his adviser, Robert Lazarsfeld for suggesting this direction of research as well as for sharing his intuition and invaluable advice. We also thank the authors of \cite{BdFF} for sharing preliminary versions of their work and address special thanks to Tommaso de Fernex for his many illuminating comments and suggestions. The author is also grateful to Bhargav Bhatt, Shihoko Ishii, Mircea Musta\c t\u a, Claudiu Raicu, Stefano Urbinati and Jonathan Wahl for useful discussions as well as to Mel Hochster for a beautiful course on the topic of local cohomology.

\section*{Notation and conventions}
\addcontentsline{toc}{section}{Notation and conventions}
Unless otherwise stated, we work over the field of complex numbers $\mathbb C$ and use the notation of \cite{L}.
For a Cartier divisor $D$ on a projective variety $X$ of dimension $n$, we consider the asymptotic cohomology functions of \cite{K}:
$$\widehat{h}^i(D)=_{\rm def}\limsup_{m\to\infty}\frac{h^i(X,\mathcal O(mD))}{m^n/n!}.$$ When $i=0$, we recover the volume function $\vol(D)$ from \cite[Ch.2.2.C]{L}.

\begin{paragraph}{The relative setting.} Let $\pi:Y\to X$ be a projective morphism of quasi-projective varieties. A Cartier divisor $D$ on $Y$ is $\pi-$\textit{trivial} if $D=\pi^*L$ for some Cartier divisor $L$ on $X$. Two Cartier divisors $D$ and $D'$ are $\pi-$\textit{linearly equivalent} if $D$ is linearly equivalent to $D'+\pi^*L$ for some Cartier divisor $L$ on $X$. A Cartier divisor $D$ on $Y$ is $\pi-$\textit{numerically trivial} if its restriction to fibers of $\pi$ is numerically trivial. The set of $\pi-$numerical equivalence classes is an abelian group of finite rank denoted $N^1(Y/X)$. A divisor $D$ is $\pi-$ample (nef) if the restriction to each fiber of $\pi$ is ample (nef). $D$ is $\pi-$movable if its $\pi-$base locus has codimension at least two in $Y$.  
\end{paragraph} 

\begin{paragraph}{Cohomology with supports.} We point to \cite{G} for an elaborate study of cohomology with supports, or \cite[Exer.III.2.3]{H} for a quick introduction that is sufficient for our purposes.
\end{paragraph}

\begin{paragraph}{Resolutions of singularities.} In a log-resolution $\pi:(\widetilde X,E)\to(X,x)$, we denote by $E$ the reduced fiber over $x$. The divisor $E$ has simple normal crossings. We say $\pi$ is a \textit{good resolution} if it is an isomorphism outside $x$.
\end{paragraph}

\begin{paragraph}{Coherent fractional ideal sheaves} A coherent subsheaf $\mathcal I$ of the constant fraction field sheaf of an integral scheme of finite type over an algebraically closed field is called a coherent fractional ideal sheaf. Typical examples are constructed by pushing forward invertible sheaves via projective birational fiber space maps. For $\mathcal I$ a coherent fractional ideal sheaf, there exists a Cartier divisor $D$ on $X$ such that $\mathcal I\cdot\mathcal O_X(D)$ is an actual ideal sheaf. Using this, the blow-up of $\mathcal I$ can be defined.\end{paragraph}

\section{Local volumes}
This section is devoted to building the theory of local volumes for Cartier divisors on a relatively projective birational modification of a normal complex quasi-projective variety of dimension at least two with a distinguished point. We compare many properties of these volumes to their counterpart in the theory of volumes of Cartier divisors on projective varieties as presented in \cite[Ch.2.2.C]{L}. In the first subsection we define the local volumes, study them variationally, discus their behavior under finite maps and give examples. In the second subsection we adapt some of the methods of \cite{LM} to present a convex body approach to local volumes and obtain a Fujita approximation result. We discuss log-convexity and vanishing properties for local volumes in our third subsection.

\subsection{Basic properties}

\par Let $X$ be a normal complex quasi-projective variety of dimension $n$ at least two over the field of complex numbers $\mathbb C$ and fix a point $x\in X$. Let $\pi:X'\to X$ be a projective birational morphism and let $D$ be a Cartier divisor on $X'$. Using cohomology with supports at $x$, define 
\begin{equation}\label{def:hl}\hl(D)=_{\rm def}\dim H^1_{\{x\}}(X,\pi_*\mathcal O_{X'}(D)).\end{equation}
We will see in the course of the proof of Proposition \ref{volfin} that this is a finite number.
\begin{rem}\label{r1}
\begin{enumerate}
\item If $U$ is an open subset of $X$ containing $x$, let $F$ be the set theoretic fiber over $x$, let $V$ be the pre-image of $U$ and denote by $i:U\setminus\{x\}\to U$  and $j:V\setminus F\to V$ the natural open embeddings. By abuse, we denote $\pi|_U^V$ again by $\pi$. An inspection of the restriction sequence for cohomology with supports, together with flat base change, reveal $$\hl(D)=\dim\frac{i_*i^*(\pi_*\mathcal O_{X'}(D)|_U)}{\pi_*\mathcal O_{X'}(D)|_U}=\dim\frac{\pi_*j_*j^*(\mathcal O_{X'}(D)|_V)}{\pi_*\mathcal O_{X'}(D)|_U}.$$
\item If $U$ is affine, $X'$ is normal and $E$ is the divisorial component of the support of the fiber, then $$\hl(D)=\dim\frac{\bigcup_{k\geq 0}H^0(\pi^{-1}U,\mathcal O_{X'}(D+kE))}{H^0(\pi^{-1}U,\mathcal O_{X'}(D))}$$
as a study of local sections shows.
\end{enumerate}
\end{rem}
\noindent  

\begin{defn} The \textit{local volume} of $D$ at $x$ is the asymptotic limit:
$$\vol_x(D)=_{\rm def}\limsup_{m\to\infty}\frac{\hl(mD)}{m^n/n!}.$$ 
\end{defn}

\noindent We will prove that this quantity is finite in Proposition \ref{volfin}. We will also see in Corollary \ref{truelim} that the $\limsup$ in the definition of $\vol_x(D)$ can be replaced by $\lim$. The excision property of cohomology with supports shows that $\vol_x$ is local around $x$. The term volume is justified by the resemblance of the definition to that of volumes of divisors on projective varieties. We shall see that the two notions share many similar properties.  

\begin{exap}[Toric varieties]\label{toric} We use the notation of \cite{F}. Let $\sigma$ be an $n-$dimensional pointed rational cone in $N_{\mathbb R}$, where $N$ is a lattice isomorphic to $\mathbb Z^n$. Denote $M=\Hom(N,\mathbb Z)$ and let $S_{\sigma}$ be the semigroup $\sigma^{\vee}\cap M$. Let $X(\sigma)$ be the affine toric variety $\Spec\mathbb C[S_{\sigma}]$. The unique torus invariant point of $X(\sigma)$ is denoted $x_{\sigma}$.
\par Let $\Sigma$ be a rational fan obtained by refining $\sigma$. It determines a proper birational toric modification $\pi:X(\Sigma)\to X(\sigma)$. Let $v_1,\ldots,v_r$ be the first non-zero integer coordinate points on the rays that span $\sigma$. Let $v_{r+1},\ldots,v_{r+s}$ be the first non-zero points of $N$ on the rays in $\Sigma$ that lie in the relative interior of faces of $\sigma$ of dimension $2\leq d\leq n-1$  and denote by $v_{r+s+1},\ldots,v_{r+s+t}$ the first non-zero points from $N$ on the rays of $\Sigma$ in the interior of $\sigma$. Denote by $D_i$ the Weil divisor on $X(\Sigma)$ associated to the ray containing $v_i$. A divisor $D_i$ lies over $x_{\sigma}$ exactly when its support is a complete variety, which is equivalent to $v_i$ lying in the interior of $\sigma$, i.e., when $i\geq r+s$.
\par To $D=\sum_{i=1}^{r+s+t}a_iD_i$, a $T-$invariant Cartier divisor on $X(\Sigma)$, we associate the rational convex polyhedra in $M_{\mathbb R}$ defined by
$$P_D=\{u\in M_{\mathbb R}:\ \langle u,v_i\rangle\geq -a_i\mbox{ for all }i\}.$$
$$P'_D=\{u\in M_{\mathbb R}:\ \langle u,v_i\rangle\geq -a_i\mbox{ for all }i\leq r+s\}.$$
By \cite[Lem.pag.66]{F}, global sections of $\mathcal O_{X(\Sigma)}(mD)$ correspond to points of $(mP_D)\cap M$ and sections defined outside the fiber over $x_{\sigma}$ correspond to $(mP'_D)\cap M$. By Remark \ref{r1},
$$h^1_{x_{\sigma}}(mD)=\#((mP'_D\setminus mP_D)\cap M).$$
Taking asymptotic limits, $$\vol_{x_{\sigma}}(D)=n!\cdot\vol(P'_D\setminus P_D).$$ On the right hand side we have the Euclidean volume in $M_{\mathbb R}$. Note that this volume is rational and finite, even though $P_D$ and $P'_D$ may be infinite polyhedra. See Example \ref{tnc} and Figure \ref{fig:toric} for an explicit computation.
\end{exap} 

The surface case, which was studied in \cite{JW1} and served as the inspiration for our work, gives another set of computable examples. 

\begin{exap}[Surface case]\label{sf1} Let $(X,x)$ be a normal isolated surface singularity and let $\pi:\widetilde X\to X$ be a good resolution. Any divisor $D$ on $\widetilde X$ admits a relative Zariski decomposition $D=P+N$ where $P$ is a relatively nef and exceptional 
$\mathbb Q-$divisor. See \cite[Section 1]{JW1} for more on relative Zariski decompositions. From \cite[Thm.1.6]{JW1}, we have 
$$\vol_x(D)=-P\cdot P$$ and this can be computed algorithmically from the dual graph of $\pi$. 
\end{exap} 

\begin{lem}\label{proj} As before, let $X$ be a normal quasiprojective variety of dimension at least two, let $x$ be a point on $X$ and let $\pi:X'\to X$ be a relatively projective birational map. Then there exist projective completions $\overline X$ and $\overline {X'}$ of $X$ and $X'$ respectively, together with a map $\overline\pi:\overline {X'}\to\overline X$ extending $\pi$ and a Cartier divisor $\overline D$ on $\overline{X'}$ such that $\overline D|_{X'}=D$.\end{lem}
\begin{proof} Choose arbitrary projective completions $\overline X$ and $Y$ of $X$ and $X'$ respectively. The rational map $\xymatrix{Y\ar@{.>}[r]&\overline X}$ induced by $\pi$ can be extended by resolving its indeterminacies in $Y$ to $\pi':Y'\to\overline X$ such that $\pi'|_{X'}=\pi$. The Cartier divisor $D$ determines an invertible sheaf $\mathcal O_{X'}(D)$, which by \cite[Exer.II.5.15]{H} extends to a coherent fractional ideal sheaf $\mathcal I$ on $Y'$ and if we denote by $\overline {X'}$ the blow-up of $Y'$ along $\mathcal I$, by $\overline\pi:\overline {X'}\to\overline X$ the induced map and by $\mathcal O_{\overline Y}(\overline D)$ the relative Serre bundle of the blow-up, one notices that $\overline D|_{X'}=D$.\end{proof}

The previous result can be used to reduce questions about the local volume of one divisor $D$ (or of finitely many) to the case when $X$ and $X'$ are projective. We will see that we can reduce the study of the function $\vol_x$ to $X'$ normal, or even non-singular.

\begin{lem}\label{l1} With notation as above, let $\mathcal F$ be a torsion free coherent sheaf on $X'$ of rank $r$. Then
$$\vol_x(D)=\limsup_{m\to\infty}\frac{\dim H^1_{\{x\}}(X,\pi_*(\mathcal F(mD)))}{r\cdot m^n/n!}.$$\end{lem}
\begin{proof}By Lemma \ref{proj}, since we can extend coherent torsion free sheaves to coherent sheaves with the same property, we can assume that $X$  and $X'$ are projective. Let $H$ be sufficiently ample on $X$ so that there exist short exact sequences 
$$0\to\mathcal O_{X'}^r(-\pi^*H)\to\mathcal F\to Q\to 0$$
$$0\to\mathcal F\to\mathcal O_{X'}^r(\pi^*H)\to R\to 0$$ 
with torsion quotients $Q$ and $R$. Such $H$ exists because $\pi^*H$ is a big Cartier divisor. 
\par If $Q_m$ and $R_m$ denote the images of $\pi_*(\mathcal F(mD))$ in $\pi_*(Q(mD))$ and of $\pi_*\mathcal O_{X'}^r(\pi^*H+mD)$ in $\pi_*(R(mD))$ respectively, then 
\begin{equation}\label{e1}\dim H^1_{\{x\}}(X,\pi_*(\mathcal F(mD)))\leq r\cdot\dim H^1_{\{x\}}(X,\pi_*\mathcal O_{X'}(\pi^*H+mD)))+\dim H^0_{\{x\}}(X,R_m),\end{equation}
$$r\cdot\dim H^1_{\{x\}}(X,\pi_*\mathcal O_{X'}(-\pi^*H+mD)))\leq \dim H^1_{\{x\}}(X,\pi_*(\mathcal F(mD)))+ \dim H^0_{\{x\}}(X,Q_m).$$
Since the cohomology of twists of torsion sheaves grows submaximally by \cite[Ex.1.2.33]{L}, from the inequality $$\dim H^0_{\{x\}}(X,Q_m)\leq \dim H^0(X,Q_m)\leq \dim H^0(X,\pi_*(Q(mD)))=\dim H^0(X',Q(mD))$$ together with the corresponding one for $R$ and (\ref{e1}), we conclude by the next easy lemma.
\end{proof}

\begin{lem}\label{invpull} 
\begin{enumerate}
\item If $L$ is a Cartier divisor on $X$, then $\hl(D+\pi^*L)=\hl(D).$ 
\item In particular, if $D$ and $D'$ are linearly equivalent on $X'$, then $\hl(D)=\hl(D').$
\end{enumerate}
\end{lem}
\begin{proof} Cohomology with supports at $x$ is a local invariant by excision. Choosing an affine neighborhood where $\mathcal O_X(L)$ is trivial yields the result.\end{proof}

\begin{cor}\label{c1} 
\begin{enumerate}
\item If $f:Y\to X'$ is projective and birational, then $\vol_x(D)=\vol_x(f^*D).$
\item The previous result holds in particular if $f$ is the normalization of $X'$, or a resolution of singularities.
\end{enumerate}
\end{cor}
\begin{proof} This is an immediate consequence of applying Lemma \ref{l1} for the torsion free sheaf of rank one $\mathcal F=f_*\mathcal O_Y$.\end{proof}

We also deduce a useful result concerning pullbacks by finite maps.
\begin{prop}\label{pfp} Let $\pi:X'\to X$ and $\rho: Y'\to Y$ be projective birational morphisms onto normal quasi-projective varieties of dimension $n$ at least two. Let $y$ be a point on $Y$. Assume $f:X\to Y$ is a finite morphism that has a lift to a generically finite morphism $f':X'\to Y'$ and let $D$ be a Cartier divisor on $Y$. Then
$$(\deg f)\cdot\vol_y(D)=\sum_{x\in f^{-1}\{y\}}\vol_x(f'^*D).$$ Note that the index family for the sum is taken set theoretically, not scheme theoretically.\end{prop}
\begin{proof} Let $i:Y\setminus\{y\}\to Y$ and $j:X\setminus f^{-1}\{y\}$ be the natural open embeddings. As a consequence of Remark \ref{r1}, $$\dim\frac{j_*j^*\pi_*\mathcal O_{X'}(f'^*D)}{\pi_*\mathcal O_{X'}(f'^*D)}=\sum_{x\in f^{-1}\{y\}}\hl(f'^*D).$$
Looking at global sections and by the finiteness of $f$,
$$\dim\frac{j_*j^*\pi_*\mathcal O_{X'}(f'^*D)}{\pi_*\mathcal O_{X'}(f'^*D)}=\dim f_*\left(\frac{j_*j^*\pi_*\mathcal O_{X'}(f'^*D)}{\pi_*\mathcal O_{X'}(f'^*D)}\right)=\dim\frac{f_*j_*j^*\pi_*\mathcal O_{X'}(f'^*D)}{f_*\pi_*\mathcal O_{X'}(f'^*D)}.$$
Chasing through the diagram
$$\xymatrix{ & X'\ar[rr]^{f'}\ar[dd]^(.3){\pi} & & Y'\ar[dd]^{\rho}\\
X''\ar@{^{(}->}[ur]^{j'}\ar[rr]^(.4){f''}\ar[dd]_{\pi'} & & Y''\ar@{^{(}->}[ur]^{i'}\ar[dd]_(.3){\rho'} & \\
 & X\ar[rr]^(.3)f & & Y\\ X\setminus f^{-1}\{y\}\ar@{^{(}->}[ur]^{j}\ar[rr]^{f'} & & Y\setminus\{y\}\ar@{^{(}->}[ur]_i & }$$ obtained by restricting outside $y$ and its pre-images, and applying flat base change (\cite[Prop. III.9.3]{H}) for the flat open embedding $i$, one finds that  $$\dim\frac{f_*j_*j^*\pi_*\mathcal O_{X'}(f'^*D)}{f_*\pi_*\mathcal O_{X'}(f'^*D)}=\dim\frac{i_*i^*\rho_*\mathcal F(D)}{\rho_*\mathcal F(D)},$$ with $\mathcal F$ denoting the rank $\deg(f)$ torsion free sheaf $f'_*\mathcal O_{X'}$ on $Y'$.  The result is now a consequence of Lemma \ref{l1} and of Corollary \ref{truelim}
\end{proof}

We are ready to study local volumes and draw parallels with the theory of volumes of Cartier divisors on projective varieties.

\begin{prop}[Finiteness]\label{volfin} If $\pi:X'\to X$ is a projective birational morphism and $D$ is a Cartier divisor on $X'$, then $\vol_x(D)$ is finite.\end{prop}
\begin{proof}We can assume that $X$ and $X'$ are projective. Choose $H$ ample on $X$ such that $\pi^*H-D$ is effective. From the restriction sequence for cohomology with supports, 
$$\hl(mD)\leq h^0(X\setminus\{x\},\pi_*\mathcal O_{X'}(mD))+h^1(X,\pi_*\mathcal O_{X'}(mD)).$$
By the choice of $H$, we have $$h^0(X\setminus\{x\},\pi_*\mathcal O_{X'}(mD))\leq h^0(X\setminus\{x\},\mathcal O_X(mH))=h^0(X,\mathcal O_X(mH)).$$ The last equality holds since $X$ is normal of dimension $n\geq 2$. For any $m\geq 0$, we have a short exact sequence
$$0\to\mathcal O_{X'}(mD)\to\mathcal O_{X'}(m\cdot\pi^*H)\to Q_m\to 0$$ that defines $Q_m$. Pushing forward and taking cohomology, one finds $$h^1(X,\pi_*\mathcal O_{X'}(mD))\leq h^0(X',Q_m)+h^1(X,\mathcal O_X(mH))\leq$$ $$\leq h^0(X,\mathcal O_X(mH))+h^1(X',\mathcal O_{X'}(mD))+h^1(X,\mathcal O_X(mH)).$$ 
We conclude that $$\vol_x(D)\leq 2\cdot\vol(H)+\widehat h^1(D)+\widehat h^1(H)$$ with the right hand side being finite by \cite[Rem.2.2]{K}.
\end{proof} 

\begin{rem}Note the when $x$ is a point on a non-singular curve, even $\dim H^1_{\{x\}}(\mathcal O_X)$ is infinite, therefore the assumption that $\dim X\geq 2$ is crucial.\end{rem}

\begin{prop}[Homogeneity]\label{homogeneity} With the same hypotheses as before, $\vol_x(mD)=m^n\cdot\vol_x(D)$ for any integer $m\geq 0$.\end{prop}
\begin{proof} Following ideas in \cite[Lem.2.2.38]{L} or \cite[Prop.2.7]{K}, for $i\in\{0,\ldots,m-1\}$, let $$a_i=_{\rm def}\limsup_{k\to\infty}\frac{\hl((mk+i)D)}{k^n/n!}.$$ It is easy to see that $$\vol_x(D)=\max_i\{\frac{a_i}{m^n}\}.$$ On the other hand, Lemma \ref{l1} implies that $a_0=\ldots=a_{m-1}=\vol_x(mD)$.
\end{proof}

Our prototype example, when we can compute local volumes and see an explicit connection to the theory of volumes of divisors on projective varieties, is the case of cones over polarized projective varieties. 

\begin{exap}[Cone singularities]\label{cone1} Let $(V,H)$ be a non-singular projective polarized variety of dimension $n-1$. The vertex $0$ is the isolated singularity of the normal variety $$X=\Spec\bigoplus_{m\geq 0}H^0(V,\mathcal O(mH)).$$ Blowing-up $0$ yields a resolution of singularities for $X$ that we denote $Y$. The induced map $\pi:Y\to X$ is isomorphic to the contraction of the zero section $E$ of the geometric vector bundle $$\Spec_{\mathcal O_V}\Sym^{\bullet}\mathcal O_V(H).$$ Let $f:Y\to V$ denote the bundle map. We have $f_*\mathcal O_Y=\Sym^{\bullet}\mathcal O_V(H)$. Being the zero section, $E$ is isomorphic to $V$. Concerning divisors on $Y$, we mention the following well known results:
$${\rm Pic}(Y)=f^*{\rm Pic}(V),$$
such that divisors on $Y$ are determined, up to linear equivalence, by their restriction to $E$: $$\mathcal O_Y(D)=f^*\mathcal O_V(D|_E).$$ The co-normal bundle of $E$ in $Y$ is: $$\mathcal O_E(-E)\simeq\mathcal O_V(H).$$
Let $L$ be a divisor on $V$ and $D=f^*L$. Since $X$ is affine, Remark \ref{r1} implies 
$${\rm h}^1_0(mD)=\dim\frac{\bigcup_{k\geq 0}H^0(Y,\mathcal O_Y(mD+kE))}{H^0(Y,\mathcal O_Y(mD))}=\sum_{k\geq 1}h^0(\mathcal O_V(mL-kH)).$$
We aim to show that
$$\vol_0(D)=n\cdot\int_0^{\infty}\vol(L-tH)dt,$$ the volume on the right hand side being the volume of Cartier divisors on the projective variety $V$. Note that the integral is actually definite, because $H$ is ample. By homogeneity and a change of variables, we can assume we are computing the integral over the interval $[0,1]$. Since $H$ is ample, the function $t\to\vol(L-tH)$ is decreasing, hence for all $k>0$, 
$$\frac 1k\cdot\sum_{i=1}^k\vol(L-\frac ikH)\leq\int_0^1\vol(L-tH)dt\leq\frac 1k\cdot\sum_{i=0}^{k-1}\vol(L-\frac ikH).$$
For any $\varepsilon>0$, there exists $s_0$ depending on $\varepsilon$ and $k$ such that for $s>s_0$,
$$\frac nk\cdot\sum_{i=0}^{k-1}\vol(L-\frac ikH)\leq \frac {n!}{k^ns^{n-1}}\cdot\sum_{i=0}^{k-1}h^0(skL-siH)+\varepsilon=$$ 
$$\frac {n!}{k^ns^{n-1}}\cdot\sum_{i=1}^{k}h^0(skL-siH)+\varepsilon+\frac{h^0(skL)-h^0(skL-skH)}{(sk)^{n-1}\cdot k/n!}\leq \frac{{\rm h}^1_0(skD)}{(sk)^n/n!}+\varepsilon+\frac{h^0(skL)-h^0(skL-skH)}{(sk)^{n-1}\cdot k/n!}.$$
Letting $s$ tend to infinity, 
$$\frac nk\cdot\sum_{i=0}^{k-1}\vol(L-\frac ikH)\leq \frac{\vol_0(kD)}{k^n}+\varepsilon+\frac{\vol(kL)-\vol(kL-kH)}{k^{n-1}\cdot k/n}=\vol_0(D)+\varepsilon +\frac{\vol(L)-\vol(L-H)}{k/n},$$
the equality taking place by the $n$ and $n-1$ homogeneity properties of $\vol_0$ and $\vol$ respectively. Taking limits with $k$ and $\varepsilon$, we obtain 
$$n\cdot\int_0^{\infty}\vol(L-tH)dt\leq\vol_0(D).$$ The reverse inequality follows in similar fashion.
\end{exap}

We say that the Weil divisor $D$ on $X'$ \textit{lies over} $x$ if $\pi(D)=\{x\}$ set theoretically, or if $D=0$. We know that the volume of Cartier divisors on projective varieties increases in effective directions and variations can be controlled  by a result of Siu (see \cite[Thm.2.2.15]{L} and \cite[Ex.2.2.23]{L}). As we shall soon see, $\vol_x$ behaves quite differently depending on whether the effective divisor lies over $x$ or if it has no components with this property. Controlling the variation of volumes in effective directions is our key to proving continuity properties. 

\begin{lem}\label{l2} On $X'$, let $E$ be an effective Cartier divisor lying over $x$. Then for any Cartier divisor $D$ on $X'$,
\begin{enumerate}
\item $\hl(D)\geq \hl(D+E)$ and hence $\vol_x(D)\geq \vol_x(D+E).$
\item $\hl(D)-\hl(D+E)\leq h^0((D+E)|_E).$
\item If $E=A-B$ with $A$ and $B$ two $\pi-$ample divisors on $X'$, then $$\vol_x(D)-\vol_x(D+E)\leq n\cdot \vol((D+A)|_E),$$ with the volume in the right hand side being the volume of divisors on the projective $n-1$ dimensional sub-scheme $E$ of $X'$.
\end{enumerate}
\end{lem}
\begin{proof}
Denote by $i$ the natural embedding $X\setminus\{x\}\hookrightarrow X$ and consider the diagram
$$\xymatrix{\pi_*\mathcal O_{X'}(D)\ar@{^{(}->}[d]\ar@{^{(}->}[r]& \pi_*\mathcal O_{X'}(D+E)\ar@{^{(}->}[d] \\
i_*i^*\pi_*\mathcal O_{X'}(D)\ar@{=}[r] & i_*i^*\pi_*\mathcal O_{X'}(D+E)}$$
We get an induced surjection between the cokernels of the vertical maps and part $(i)$ follows by Remark \ref{r1}. The same remark, together with the inclusion map $$\frac{\pi_*\mathcal O_{X'}(D+E)}{\pi_*\mathcal O_{X'}(D)}\hookrightarrow\pi_*\mathcal O_E(D+E)$$ lead to part $(ii)$. A repeated application of $(ii)$ yields 
$$\hl(mD)-\hl(mD+mE)\leq\sum_{k=1}^{m}h^0((mD+kE)|_E)\leq m\cdot h^0(m(D+A)|_E),$$ with the last inequality following from the assumptions on $A$ and $B$ that imply the effectiveness of $A|_E$ and $(A-E)|_E$. Part $(iii)$ follows by taking asymptotic limits.
\end{proof}

\noindent Quite opposite behavior is observed for effective divisors without components over $x$. We can control variations in such directions only in the non-singular case, but we do have the tools to reduce our general questions to this case.

\begin{lem}\label{l3} Assume $X'$ is non-singular and let $F$ be an effective divisor without components lying over $x$. There exists a $\pi-$ample divisor $-\Delta_1-\Delta_2$ with $\Delta_1$ effective lying over $x$ and $\Delta_2$ effective without components over $x$, such that $-\Delta_1-\Delta_2-F$ is $\pi-$very ample. Write $\Delta_1=M-N$ with $M$ and $N$ two $\pi-$ample divisors. Then for any divisor $D$,
\begin{enumerate}
\item $\hl(D+F)\geq\hl(D)$ and $\vol_x(D+F)\geq\vol_x(D)$.
\item $\hl(D+F)-\hl(D)\leq h^0(D|_{\Delta_1})$.
\item $\vol_x(D+F)-\vol_x(D)\leq n\cdot\vol((D+N)|_{\Delta_1}).$
\end{enumerate}
\end{lem}
\begin{proof} To justify the existence of $\Delta_1$ and $\Delta_2$, it is enough to show that there exists a negative (its dual is effective) $\pi-$ample divisor. By \cite[Thm.II.7.17]{H}, since $\pi$ is projective birational, $X'$ is the blow-up of some ideal sheaf on $X$. The relative Serre bundle of the blow-up is both negative and $\pi-$ample.
\par Let $i$ be the natural open embedding $X\setminus\{x\}\hookrightarrow X$. Examining the diagram
$$\xymatrix{\pi_*\mathcal O_{X'}(D)\ar@{^{(}->}[d]\ar@{^{(}->}[r]& \pi_*\mathcal O_{X'}(D+F)\ar@{^{(}->}[d] \\
i_*i^*\pi_*\mathcal O_{X'}(D)\ar@{^{(}->}[r] & i_*i^*\pi_*\mathcal O_{X'}(D+F)}$$
we get an induced injective morphism between the cokernels of the vertical maps if we show that
$$\pi_*\mathcal O_{X'}(D)=\pi_*\mathcal O_{X'}(D+F)\cap i_*i^*\pi_*\mathcal O_{X'}(D),$$ the intersection taking place in $i_*i^*\pi_*\mathcal O_{X'}(D+F)$. It is enough to show this on the level of sections over open neighborhoods of $x$. Let $U$ be such an open set on $X$ and let $V$ be its inverse image in $X'$. Let $E$ be the divisorial support of the set theoretic fiber $\pi^{-1}(x)$. Since $X'$ is in particular normal, we have to show $$H^0(V,\mathcal O_{X'}(D))=H^0(V,\mathcal O_{X'}(D+F))\cap H^0(V\setminus\{E\},\mathcal O_{X'}(D))$$ inside $H^0(V\setminus\{E\},\mathcal O_{X'}(D+F))$ which is easily checked. Part $(i)$ follows by Remark \ref{r1}.
\vskip .3cm
\par Let $A$ be a divisor without components over $x$ that is $\pi-$linearly equivalent to $-\Delta_1-\Delta_2-F$. By part $(i)$ and Lemmas \ref{invpull} and \ref{l2}, 
$$\hl(D+F)-\hl(D)\leq\hl(D+F+A+\Delta_2)-\hl(D)=\hl(D-\Delta_1)-\hl(D)\leq h^0(D|_{\Delta_1}).$$ Part $(iii)$ follows similarly. 
\end{proof}

\noindent We aim to prove that $\vol_x(D)$ depends only on the $\pi-$relative numerical class of $D$ in $N^1(X'/X)$.

\begin{lem}\label{l4} Let $N$ be a $\pi-$nef divisor on $X'$. Then $\vol_x(D+N)\geq\vol_x(D).$\end{lem}
\begin{proof}
By Lemma \ref{l1}, we can assume that $X'$ is non-singular. Let then $F$ be a $\pi-$ample divisor on $X'$. For any $m\geq 1$, there exists $k_m>0$ such that $k_m(mN+F)$ is $\pi-$linearly equivalent to an effective divisor without components lying over $x$. By Lemmas  \ref{invpull} and \ref{l3} and Proposition \ref{homogeneity}, we have 
\begin{equation}\label{e2}\frac{\vol_x(m(D+N)+F)}{m^n}\geq \vol_x(D).\end{equation}
By part $(iii)$ of Lemma \ref{l3}, with the notation there,
$$\vol_x(m(D+N)+F)-\vol_x(m(D+N))\leq n\cdot\vol((m(D+N)-\Delta_1+M)|_{\Delta_1}).$$
Since the support of $\Delta_1$ is of dimension $n-1$, dividing by $m^n$ and applying Proposition \ref{homogeneity} and the inequality (\ref{e2}), 
$$\vol_x(D+N)=\lim_{m\to\infty}\frac{\vol_x(m(D+N)+F)}{m^n}\geq\vol_x(D).$$ 
\end{proof}

\begin{cor}[Relative numerical invariance]\label{numinv} Let $N$ be a $\pi-$numerically trivial divisor on $X'$. Then for any Cartier divisor $D$ on $X'$, we have $$\vol_x(D+N)=\vol_x(D).$$\end{cor}
\begin{proof}
Both $N$ and $-N$ are $\pi-$nef, hence $$\vol_x(D)\leq\vol_x(D+N)\leq\vol_x((D+N)+(-N))=\vol_x(D).$$
\end{proof}

\noindent By Corollary \ref{numinv}, the local volume $\vol_x$ is a well defined function on $N^1(X'/X)$. From the homogeneity result in Proposition \ref{homogeneity}, it also has a natural extension to $N^1(X'/X)_{\mathbb Q}$. By proving continuity on this space, we are able to extend to real coefficients.

\begin{prop}[Continuity]\label{cont} Let $\pi:X'\to X$ be a projective birational morphism to a normal quasi-projective variety $X$ of dimension $n$ and let $x$ be a point on $X$. The relative numerical real space $N^1(X'/X)_{\mathbb R}$ is finite dimensional and fix a norm $|\cdot|$ on it. Then there exists a positive constant $C$ such that for any $A$ and $B$ in the rational vector space $N^1(X'/X)_{\mathbb Q}$ we have the estimate: $$|\vol_x(B)-\vol_x(A)|\leq C\cdot (\max(|A|,|B|))^{n-1}\cdot|A-B|.$$
\end{prop}
\begin{proof} We show we can assume that $X'$ is non-singular. Let $f:Y\to X'$ be a resolution of singularities. Then $f^*$ induces an injective morphism $N^1(X'/X)\hookrightarrow N^1(Y/X)$ which does not change the values of $\vol_x$ by Corollary \ref{invpull}. Hence it is sufficient to prove our estimate for $X'$ non-singular.
\par We can choose $\lambda_1,\ldots,\lambda_k$ a basis for $N^1(X'/X)_{\mathbb R}$ composed of integral $\pi-$very ample divisors without components over $x$. Relative to this basis, we can assume that $$|(a_1,\ldots,a_k)|=\max_{1\leq i\leq k}|a_i|.$$ With notation as in Lemma \ref{l3}, choose $\Delta_1$ and $\Delta_2$ two effective integral divisors with the first lying over $x$ whereas the second has no components over $x$ such that for all $i\in\{1,\ldots,k\}$, the divisor $-\Delta_1-\Delta_2-\lambda_i$ is $\pi-$linearly equivalent to one without components lying over $x$. Write $\Delta_1=M-N$ with $M$ and $N$ two $\pi-$ample divisors. Let $$A=(a_1,\ldots,a_k),\hspace{.5 in} B=(a_1+b_1,\ldots,a_k+b_k),\hspace{.5 in} N=(\alpha_1,\ldots,\alpha_k)$$ with all entries being rationals. Since our estimate to prove and $\vol_x$ are both $n-$homogeneous, we can further assume that all the entries are integers. Note that the $\alpha_i$ are fixed. 
\par If we denote $B_i=(a_1,\ldots,a_i,a_{i+1}+b_{i+1},\ldots,a_k+b_k)$ and set
$$A_i=\left\{\begin{array}{ll}B_{i-1}+b_iN,& \mbox{if }b_i\geq 0\\ B_i-b_iN,& \mbox{if }b_i\leq 0\end{array}\right.,$$ then
$$|\vol_x(B)-\vol_x(A)|\leq\sum_{i=1}^k|\vol_x(B_{i-1})-\vol_x(B_i)|\leq n\cdot\sum_{i=1}^k\vol(A_i|_{|b_i|\Delta_1})$$ by Lemma \ref{l3}. Let $$\alpha=\max_{1\leq i\leq k}(|\alpha_i|).$$ Since $\lambda_i|_{|b_i|\Delta_1}$ is ample for all $i$ and $\vol(D|_{\Delta_1})=D^{n-1}\cdot \Delta_1$ if $D$ is $\pi-$ample,
$$n\cdot\sum_{i=1}^k\vol(A_i|_{|b_i|\Delta_1})\leq n(1+\alpha)^{n-1}\cdot\max_{1\leq i\leq k}(|a_i|+|b_i|)^{n-1}\cdot ((\sum_{i=1}^k\lambda_i)^{n-1}\cdot\Delta_1)\cdot\sum_{i=1}^k|b_i|.$$ Setting $$C=nk\cdot 2^{n-1}(1+\alpha)^{n-1}\cdot ((\sum_{i=1}^k\lambda_i)^{n-1}\cdot\Delta_1)$$ concludes the proof.
\end{proof}

\noindent Putting together Propositions \ref{homogeneity} and \ref{cont} with Corollary \ref{numinv}, we have proved:

\begin{thm}\label{t1} Let $X$ be a normal complex quasi-projective variety of dimension $n$ and let $x$ be a point on $X$. Let $\pi:X'\to X$ be a projective birational morphism. Then $\vol_x$ is a well defined, $n-$homogeneous and continuous function on $N^1(X'/X)_{\mathbb R}$.\end{thm}

We say a few words about extending the results in this subsection to proper birational morphisms and to algebraically closed fields of arbitrary characteristic. 
\begin{rem}\label{nonqproj}By working in an affine neighborhood of $x\in X$, we can remove the assumption that $X$ is quasi-projective. Our only global result is Theorem \ref{t1} whose hypothesis restrictions come from the local statements.\end{rem}
\begin{rem}[Proper morphisms]\label{propermorph}Using Chow's lemma (\cite[Ex.II.4.10]{H}) and adjusting the proof of Lemma \ref{l1}, we can extend our results to proper birational morphisms $\pi:X'\to X$.\end{rem}
\begin{rem}[Generically finite morphisms]\label{genfinmorph} Let $\pi:X'\to X$ be a generically finite proper morphism with $X$ normal, let $D$ be a Cartier divisor on $X'$ and let $x$ be a point on $X$. Denote by $\widetilde X$ the normalization of $X'$, by $\widetilde D$ the lift of $D$ and by $\widetilde Y$ the normalization of the Stein factorization (\cite[Cor.III.11.5]{H}) of $\pi$. Note that $\widetilde Y$ is the Stein factorization of the induced map $\widetilde X\to X$ and that the map $\widetilde X\to\widetilde Y$ is birational. Let $\{y_1,\ldots,y_k\}$ be the set theoretic pre-image of $x$ in $\widetilde Y$. Then one can define 
$$\vol_x(D)=_{\rm def}\frac 1{\deg\pi}\sum_{i=1}^k\vol_{y_i}(\widetilde D).$$ Proposition \ref{pfp} and Lemma \ref{l1} make this definition compatible with the birational case, i.e.,
$$\vol_x(D)=\limsup_{m\to\infty}\frac{\dim H^1_{\{x\}}(X,\pi_*\mathcal O_{X'}(mD))}{\deg(\pi)\cdot m^n/n!}.$$
\end{rem}
\begin{rem}[Positive characteristic]\label{charp} We have used characteristic $0$ in studying the variational behavior of local volumes in Lemma \ref{l3} where we reduced to $X'$ being non-singular, which we could upon replacing $X'$ by a resolution of singularities. In arbitrary characteristic, over an algebraically closed field, to extend the results of this subsection, one first replaces $X'$ by a regular alteration (\cite{dJ}) and applies the discussion above for generically finite proper morphisms to reduce to the case where $\pi$ is birational and $X'$ is regular. The price to pay is that $x$ is replaced by a collection of finite points, but this is afforded by Proposition \ref{pfp} via Corollary \ref{truelim} which extends in characteristic $p$ under the assumption that $X'$ is regular.
\end{rem}

\subsection{Convex bodies and Fujita approximation}

Given a projective birational morphism $\pi:X'\to X$ onto the complex normal algebraic variety $X$ of dimension $n\geq 2$ and given $x\in X$, for any Cartier divisor $D$ on $X'$, we realize $\vol_x(D)$ as a volume of a not necessarily convex polytope arising naturally as the bounded difference of two possibly infinite convex nested polyhedra. This approach has proven effective in \cite{LM} in particular for proving that volumes of Cartier divisors are actual limits and for developing Fujita approximation type results. By employing similar techniques, we extend these results to the local setting.

\par Assume unless otherwise stated that $\pi:(X',E)\to(X,x)$ is a log-resolution of the normal affine $X$, with $x$ not necessarily an isolated singularity, and let $E=E_1+\ldots+E_k$ be the irreducible decomposition of the reduced fiber over $x$. Since $X$ is assumed to be affine, for any divisor $D$ on $X'$, we have by Remark \ref{r1} that
$$H^1_{\{x\}}(X,\pi_*\mathcal O_{X'}(D))=\frac{H^0(X'\setminus E,\mathcal O_{X'}(D))}{H^0(X',\mathcal O_{X'}(D))}.$$

The dimension of the above vector space is $\hl(D)$ (see \ref{def:hl}). Spaces of sections of multiples of line bundles on $X'$ are studied in \cite{LM} via valuation like functions defined with respect to a choice of a complete flag. It is important to work with line bundles on $X'$ and not $X'\setminus E$. In this regard, the following lemma helps us handle $H^0(X'\setminus E,\mathcal O_{X'}(mD))$ for all $m\geq 0$.

\begin{lem}\label{linbd1}In the above setting, for any divisor $D$ on non-singular $X'$ there exists $r>0$ such that for all $m\geq 0$ there is a natural identification $$H^0(X'\setminus E,\mathcal O_{X'}(mD))\simeq H^0(X',\mathcal O_{X'}(m(D+rE))).$$\end{lem}
\begin{proof}
For any divisor $L$ on $X'$, identify 
\begin{equation}\label{identify}H^0(X',\mathcal O_{X'}(L))=\{f\in K(X)\ :\ {\rm div}(f)+L\geq 0\}.\end{equation}
With this identification, recall that $$H^0(X'\setminus E,\mathcal O_{X'}(mD))=\bigcup_{i\geq0}H^0(X',\mathcal O_{X'}(mD+iE)).$$ 
There exists an inclusion $\mathcal O_{X'}(D)\subseteq\pi^*\mathcal O_X(H)$ for some effective Cartier (sufficiently ample) divisor $H$ on $X$. Since $X$ is normal, rational functions defined outside subsets of codimension two or more extend and so $$H^0(X'\setminus E,\pi^*\mathcal O_X(mH))=H^0(X\setminus\{x\},\mathcal O_X(mH))=H^0(X,\mathcal O_X(mH))=H^0(X',\pi^*\mathcal O_X(mH)).$$ 
\noindent For all non-negative $i$ and $m$, the following natural inclusions are then equalities:
$$H^0(X',\mathcal O_{X'}(\pi^*mH))\subseteq H^0(X',\mathcal O_{X'}(\pi^*mH+iE))\subseteq H^0(X'\setminus E,\mathcal O_{X'}(\pi^*mH)).$$ 
Choose $r$ so that the order of $D+rE$ along any irreducible component of $E$ is strictly greater than the order of $\pi^*H$ along the same component. 
For $s>r$, that ${\rm div}(f)+m(D+sE)$ is effective implies that $$f\in H^0(X',\mathcal O_{X'}(m(D+sE)))\subseteq H^0(X',\mathcal O_{X'}(m(\pi^*H+sE)))=H^0(X',\mathcal O_{X'}(\pi^*mH)),$$ therefore ${\rm div}(f)+\pi^*mH$ is also effective. Looking at the orders along the components of $E$, because of our choice of $r$, we actually get $f\in H^0(X',\mathcal O_{X'}(m(D+rE)))$.
\end{proof}

\par Consider a complete flag of subvarieties of $X'$, i.e., each is a divisor in the previous subvariety:
$$Y_{\bullet}:\ X'=Y_0\supset E_1=Y_1\supset\ldots\supset Y_n=\{y\}$$ such that each $Y_i$ is non-singular at the generic point of $Y_{i+1}$. Recall that $E_1$ is a component of $E$, the reduced fiber of $\pi$ over $x$. Following \cite[1.1]{LM}, for any divisor $D$ on $X'$, we construct a valuation like function $$\nu=\nu_D=(\nu_1,\ldots,\nu_n):H^0(X',\mathcal O_{X'}(D))\to\mathbb Z^{n}\cup\{\infty\}$$ having the following properties:

\begin{align*}
(i).\ & \nu(s)=\infty\mbox{ if, and only if, }s=0.\\
(ii).\ & \nu(s+s')\geq\min\{\nu(s),\nu(s')\}\mbox{ for any }s,s'\in H^0(X',\mathcal O_{X'}(D)).\\
(iii).\ & \nu_{D_1+D_2}(s_1\otimes s_2)=\nu_{D_1}(s_1)+\nu_{D_2}(s_2)\mbox{ for any divisors }D_i\mbox{ on }X'\mbox{ and any }s_i\in H^0(\mathcal O_{X'}(D_i)).
\end{align*} 

Each $\nu_i$ is constructed by studying orders of vanishing along the terms of the flag $Y_{\bullet}$. For $s\in H^0(X',\mathcal O_{X'}(D))$, define first $\nu_1(s)$ as the order of vanishing of $s$ along $E_1$. If $f$ is the rational function corresponding to $s$ via the identification (\ref{identify}), then $\nu_1(s)$ is the coefficient of $E_1$ in ${\rm div}(f)+D$. A non-unique local equation for $Y_1$ in $Y_0$ then determines a section $$\overline s\in H^0(Y_1,\mathcal O_{Y_0}(D-\nu_1(s)Y_1)|_{Y_1})$$ having a uniquely defined order of vanishing along $Y_2$ that we denote $\nu_2(s)$ and the construction continues inductively. More details can be found in \cite[1.1]{LM}. Note that the $\nu_i$ assume only non-negative values.
\par For any divisor $D$ on $X'$ and for $m\geq 0$, with $r$ given by Lemma \ref{linbd1}, let  

\begin{align}
& I'_m =\nu_{m(D+rE)}(H^0(X',\mathcal O_{X'}(m(D+rE)))),\\ 
& I_m =\nu_{m(D+rE)}(H^0(X',\mathcal O_{X'}(mD))),\\ 
&\label{B} B_m=I'_m\setminus I_m.
\end{align}

By construction, $I'_{\bullet}=\bigcup_{m\geq0}(I'_m,m)$ and $I_{\bullet}=\bigcup_{m\geq0}(I_m,m)$ are semigroups of $\mathbb N^{n+1}$. We abuse notation in identifying the sets $I_m$ and $(I_m,m)$. We will soon prove (Lemma \ref{promise}) that $$\#B_m=\dim\frac{H^0(X',\mathcal O_{X'}(m(D+rE)))}{H^0(X',\mathcal O_{X'}(mD))}=\hl(mD).$$
Assuming this result, we aim to show that $\vol_x(D)$ is the normalized volume of the not necessarily convex polytope $B$ obtained as the difference of two nested polytopes arising as Okounkov bodies of some sub-semigroups of $I'_{\bullet}$ and $I_{\bullet}$ respectively, each satisfying the conditions \cite[(2.3)-(2.5)]{LM}. For a semigroup $\Gamma_{\bullet}\subseteq \mathbb N^{n+1}$ with $\Gamma_m=\Gamma_{\bullet}\cap (\mathbb N^n\times\{m\})$, these conditions are as follows:

\begin{align*}
\mbox{(Strictness): }& \Gamma_0=\{0\}.\\
\mbox{(Boundedness): }& \Gamma_{\bullet}\subseteq\Theta_{\bullet},\mbox{ for some semigroup }\Theta_{\bullet}\subseteq \mathbb N^{n+1} \mbox{ generated by the finite set }\Theta_1.\\
\mbox{(Denseness): }& \Gamma_{\bullet}\mbox{ generates }\mathbb Z^{n+1}\mbox{ as a group.}
\end{align*}
A semigroup $\Gamma_{\bullet}$ satisfying the above conditions spans the cone $$\Sigma(\Gamma)\subset\mathbb R^{n+1}_{\geq0}$$ which determines the convex polytope (the associated Okounkov body) $$\Delta(\Gamma)=\Sigma(\Gamma)\cap(\mathbb R^n\times\{1\}).$$ By \cite[Prop.2.1]{LM}, with the volume on $\mathbb R^n$ normalized so that the volume of the unit cube is one,
$$\vol_{\mathbb R^n}(\Delta(\Gamma))=\lim_{m\to\infty}\frac{\#\Gamma_m}{m^n}.$$
Our first challenge is to show that $B_m$ (see \ref{B}) is linearly bounded with $m$. With Lemma \ref{promise} still to prove, we show the following apparently stronger independent result:

\begin{lem}\label{linbd2}For a divisor $D$ on non-singular $X'$, with $r$ as in Lemma \ref{linbd1}, there exists $N>0$ such that for all $i$ and $m$, with valuation like functions on $H^0(X',\mathcal O_{X'}(m(D+rE)))$ as above, we have $\nu_i(s)\leq mN$ for any $s\in H^0(X',\mathcal O_{X'}(m(D+rE)))\setminus H^0(X',\mathcal O_{X'}(mD))$, e.g., $\nu_{m(D+rE)}(s)\in B_m$.\end{lem}
\begin{proof}
Let $H$ be a relatively ample integral divisor on $X'$ and assume we have shown that there exists such a linear bound $N_1$ for $\nu_1$. Since $Y_1$ is projective, as in \cite[Lem.1.10]{LM}, there exists $N_2$ such such that for all real number $0\leq a<N_1$ $$((D+rE-aY_1)|_{Y_1}-N_2Y_2)\cdot H^{n-2}<0.$$ This provides the linear bound for $\nu_2$ and one iterates this construction for all $i>1$. Letting $N$ be the maximum of all $N_i$ completes the proof. We still have to construct $N_1$. The idea here is to apply a theorem of Izumi that shows that a regular function with a high order of vanishing along $E_1$ also vanishes to high order along the other $E_i$. The technical part is to see how to apply this to rational functions giving sections of $\mathcal O_{X'}(m(D+rE))$. Since $X$ is assumed to be affine, there exists a rational function $g$ such that $$G=_{\rm def}{\rm div}(g)-D-rE$$ is effective on $X'$. With the identification in (\ref{identify}), for any $f$ in $H^0(X',\mathcal O_{X'}(m(D+rE)))$, the a priori rational function $f\cdot g^m$ is regular on $X'$. Let 
$${\rm div}(f\cdot g^m)=C+\sum_{i=1}^kc_iE_i$$ with $E_i$ the components of the reduced fiber $E$ over $x$, with $c_i\geq 0$ for all $i$ and $C$ an effective divisor without components over $x$. There exists $R>1$ such that if $c_1>0$, then $$R>\frac{c_i}{c_j}>\frac 1R$$ for all $i,j$. This is an analytic result of Izumi (\cite{Iz}), extended to arbitrary characteristic by Rees (\cite{R}). It follows that even when $c_1=0$,
$${\rm div}(f\cdot g^m)=C+\sum_{i=1}^kc_iE_i\geq C+\frac{c_1}R\cdot E.$$
If $s$ is the regular section associated to $f$, i.e., its zero locus is $Z(s)={\rm div}(f)+m(D+rE)$, then the above inequality can be rewritten as
$$Z(s)=C-mG+\sum_{i=1}^kc_iE_i\geq C-mG+\frac{c_1}R\cdot E.$$
If $\rho$ is the maximal coefficient of any $E_i$ in $G$ and $g_1$ is the coefficient of $E_1$, we set $N_1=R(r+\rho-g_1)$ and see that when $\nu_1(s)=c_1-mg_1>mN_1$, then $Z(s)\geq mrE$ showing that $$s\in H^0(X',\mathcal O_{X'}(mD))\subseteq H^0(X',\mathcal O_{X'}(m(D+rE))).$$
\end{proof}

We now prove that $B_m$ has the expected cardinality.

\begin{lem}\label{promise}With notation as above, for all $m\geq 0$, we have $\#B_m=\hl(mD)$.\end{lem}
\begin{proof}Without loss of generality, we can assume that $m=1$. Denote $B=_{\rm def}B_1$. By Lemma \ref{linbd2}, the set $B$ is bounded and therefore finite. The idea is to reduce the problem to the projective setting where we apply \cite[Lem.1.3]{LM}.
\par Recall that $X$ is assumed to be affine. Let $\overline{\pi}:\overline{X'}\to\overline{X}$ be a compactification of $\pi$ such that $\overline{X}\setminus X$ is the support of an ample divisor $H$. By abuse, we use the same notation for $H$ and its pullback and we use the same notation for $D$ and its closure in $\overline{X'}$. Note that the pullback of $H$ is big and semi-ample. For all $m\geq0$, the natural inclusion 
$$H^0(\overline{X'},\mathcal O_{\overline{X'}}(m(tH+D+rE)))\subset H^0(X',\mathcal O_{X'}(m(D+rE)))$$ is compatible with the valuation like functions $\nu_{m(tH+D+rE)}$ and $\nu_{m(D+rE)}$ that we construct when working over $\overline{X'}$ and $X'$ respectively with the flag $Y_{\bullet}$ and the obvious compactification that replaces $Y_0=X'$ by $\overline{X'}$ and leaves the remaining terms unchanged. We have the same compatibility for $\nu_{m(tH+D)}$ and $\nu_{mD}$. Note also that 
$$H^0(X',\mathcal O_{X'}(D+rE))=\bigcup_{t\geq0}H^0(\overline{X'},\mathcal O_{\overline{X'}}(tH+D+rE))$$ and a similar statement holds for $D$. When $t$ is sufficiently large so that 
\begin{equation}\label{h1van}H^1(\overline{X},\overline{\pi}_*\mathcal O_{\overline{X'}}(D)\otimes\mathcal O_{\overline X}(tH))=0,\end{equation}
excision and the natural cohomology sequence on $\overline{X}$ show that
\begin{equation}\label{setdif}H^1_{\{x\}}(X,\pi_*\mathcal O_{X'}(D))\simeq \frac{H^0(\overline{X'},\mathcal O_{\overline{X'}}(tH+D+rE))}{H^0(\overline{X'},\mathcal O_{\overline{X'}}(tH+D))}.\end{equation}
Note that the $r$ provided by Lemma \ref{linbd1} also works to prove $$H^0(\overline{X'},\mathcal O_{\overline{X'}}(tH+D+rE))=H^0(\overline{X'}\setminus E,\mathcal O_{\overline{X'}}(tH+D)).$$ Denote 
\begin{align*}
W'_t& =H^0(\overline{X'},\mathcal O_{\overline{X'}}(tH+D+rE))\\
W_t& =H^0(\overline{X'},\mathcal O_{\overline{X'}}(tH+D))\\
W'& =\bigcup_{t\geq 0}W'_t=H^0(X',\mathcal O_{X'}(D+rE))\\
W& =\bigcup_{t\geq 0}W_t=H^0(X',\mathcal O_{X'}(D))
\end{align*}
With the intersection taking place in $W'$, note that $$W_t=W'_t\cap W.$$
Let $t$ be large enough so that the vanishing (\ref{h1van}) takes place and such that $\nu(W'_t)$ contains the set $\mathcal N$ all elements in $\nu(W')$ satisfying the bound in Lemma \ref{linbd2} and such that $\nu(W_t)$ contains all elements in $\nu(W)\cap\mathcal N$. We show that $$\nu(W'_t)\setminus\nu(W_t)=\nu(W')\setminus\nu(W)=B.$$ Since $B\subset\mathcal N$ by Lemma \ref{linbd2}, all its elements are in $\nu(W'_t)$ by the choice of $t$ and are not in $\nu(W_t)\subset\nu(W)$. Therefore $B\subseteq\nu(W'_t)\setminus\nu(W_t)$. Again by the choice of $t$, any element in $ \nu(W'_t)\setminus\nu(W_t)$ that is not in $B$ is also not in $\mathcal N$. Let $\sigma\in W'_t$ such that $\nu(\sigma)\in(\nu(W'_t)\setminus\nu(W_t))\setminus B.$ Then $\nu(\sigma)\in\nu(W'_t)\setminus\mathcal N$ and again by Lemma \ref{linbd2} we obtain $\sigma\in W$, hence $\sigma\in W\cap W'_t=W_t$ and $\nu(\sigma)\in\nu(W_t)$ which is impossible. 
\par Now $\#B=\#(\nu(W'_t)\setminus\nu(W_t))=\#\nu(W'_t)-\#\nu(W_t)=\hl(D)$ by (\ref{setdif}) and by \cite[Lem.1.3]{LM}, a result that shows $\#\nu(W'_t)=\dim W'_t$ and the analogous result for $W_t$.
\end{proof}

We next construct sub-semigroups $\Gamma'_{\bullet}\subset I'_{\bullet}$ and $\Gamma_{\bullet}\subset I_{\bullet}$, each satisfying the properties \cite[(2.3)-(2.5)]{LM} and such that $B_m=\Gamma'_m\setminus\Gamma_m$ for all $m$.  With notation as in the proof of Lemma \ref{promise} and with $t$ sufficiently large so that $tH+D$ is big, let 
$$S_m=\nu_{m(tH+D+rE)}(H^0(\overline{X'},\mathcal O_{\overline{X'}}(m(tH+D)))).$$
If we pick the flag $Y_{\bullet}$ so that $Y_n=\{y\}$ is not contained in any $E_i$ for $i>1$ (we thank T. de Fernex for suggesting this choice), then 
$$S_m=\mbox{translation of }\nu_{m(tH+D)}(H^0(\overline{X'},\mathcal O_{\overline{X'}}(m(tH+D))))\mbox{ by }(mr,0,0\ldots,0,m)\in\mathbb N^{n+1}.$$
That $S_{\bullet}$ satisfies the conditions \cite[(2.3)-(2.5)]{LM} follows by \cite[Lem.2.2]{LM}. By \cite[Lem.1.10]{LM}, there exists a linear bound for $S_{\bullet}$ in the sense of Lemma \ref{linbd2}. Let $N$ be the greatest of the two linear bounds provided by Lemmas \ref{linbd2} and \cite[Lem.1.10]{LM}. If $x_i$ denotes the $i$-th coordinate on $\mathbb N^n$, let 
$$\Gamma_m=\{(x_1,\ldots,x_n)\in I_m\ :\ x_i\leq mN\ \mbox{ for all }1\leq i\leq n\}$$ and construct $\Gamma'_{\bullet}$ similarly. By construction, these semigroups satisfy the strictness (\cite[(2.3)]{LM}) and boundedness (\cite[(2.4)]{LM}) conditions. They also each generate $\mathbb Z^n$ as a group because they contain $S_{\bullet}$ which does. By Lemma \ref{linbd2}, we have $B_m=\Gamma'_m\setminus\Gamma_m$. Letting $B=\Delta(\Gamma')\setminus\Delta(\Gamma)$, we prove:

\begin{prop}\label{limconv}With notation as above, for a volume function on $\mathbb R^n$ normalized so that the volume of the unit cube is one,
$$\vol_x(D)=n!\cdot\vol_{\mathbb R^n}(B).$$\end{prop} 
\begin{proof}
$$\vol_x(D)=\limsup_{m\to\infty}\frac{\hl(mD)}{m^n/n!}=n!\cdot\limsup_{m\to\infty}\frac{\# B_m}{m^n}=n!\cdot\limsup_{m\to\infty}\frac{\#\Gamma'_m-\#\Gamma_m}{m^n}.$$
By \cite[Prop.2.1]{LM}, $$\lim_{m\to\infty}\frac{\#\Gamma'_m-\#\Gamma_m}{m^n}=\vol_{\mathbb R^n}(\Delta(\Gamma')\setminus\Delta(\Gamma))=\vol_{\mathbb R^n}(B)$$ and the conclusion follows.
\end{proof}

As a corollary we obtain that the $\limsup$ in the definition of $\vol_x$ can be replaced by $\lim$ in the general case.
\begin{cor}\label{truelim} Let $\pi:X'\to X$ be a projective birational morphism onto the complex normal algebraic $X$ of dimension $n\geq 2$ and let $x$ be a point on $X$. Then for any Cartier divisor on $X'$, we have $$\vol_x(D)=\lim_{m\to\infty}\frac{\hl(mD)}{m^n/n!}.$$\end{cor}
\begin{proof}Let $f:\widetilde X\to X'$ be a projective birational morphism such that $\rho=\pi\circ f:\widetilde X\to X$ is a log-resolution of $(X,x)$. Since $\vol_x(D)$ is local around $x$, we can also assume that $X$ is affine. By the proof of Lemma \ref{l1}, the sequences $\hl(mD)$ and $\hl(mf^*D)$ have the same asymptotic behavior. Therefore we have reduced to the setting of Proposition \ref{limconv} whose proof shows the result.\end{proof}

\begin{rem}The natural approach to the problem of expressing $\vol_x(D)$ as a volume of a polytope and replacing $\limsup$ by $\lim$ is to write $B_m=\Gamma'_m\setminus\Gamma_m$ with $\Gamma'_{\bullet}$ and $\Gamma_{\bullet}$ semigroups constructed on compacfications of $\pi$, in the same style as we did for $S_{\bullet}$, and then apply \cite[Thm.2.13]{LM}. This approach is successful when we have an analogue of \cite[Lem.3.9]{LM}, i.e., when we can show that, at least asymptotically, the groups $H^1(\overline X,\overline{\pi}_*\mathcal O_{\overline{X'}}(mD)\otimes\mathcal O_{\overline{X}}(mH))$ vanish for some ample divisor $H$ on a projective compactification $\overline{\pi}$ of $\pi$. We do not know if such a result holds for any Cartier divisor $D$ on $X'$, but we do know it when the graded family $\mathfrak a_m=\pi_*\mathcal O_{X'}(mD)$ is divisorial outside $x$, e.g., when $D$ lies over $x$, or when $D=K_{\widetilde X}+aE$ with $a\in\mathbb Z$ on a log-resolution $\pi:(\widetilde X,E)\to(X,x)$ of a normal isolated singularity.
\end{rem}

\par The content of the classical Fujita approximation statement is that the volume of an integral Cartier divisor $D$ on a projective variety $X$ of dimension $n$ can be approximated arbitrarily close by volumes $\vol(A)$ where $A$ is a nef Cartier $\mathbb Q-$divisor on some blow-up of $\pi:X'\to X$ such that $A\leq\pi^*D$. In the local setting, a step forward in proving a similar result is provided by \cite[Thm.3.8]{LM}, but before discussing it we introduce some notation.

\begin{defn}\label{SFuj} Let $X$ be a quasi-projective variety of dimension $n$ with a fixed point $x$. On $X$, consider a graded sequence of fractional ideal sheaves $\mathfrak a_{\bullet}$ (i.e. $\mathfrak a_0=\mathcal O_X$ and $\mathfrak a_k\cdot\mathfrak a_l\subseteq\mathfrak a_{k+l}$) and define its generalized Hilbert-Samuel multiplicity at $x$ as $$\mult(\mathfrak a_{\bullet})=_{\rm def}\limsup_{p\to\infty}\frac{\dim H^1_{\{x\}}(\mathfrak a_p)}{p^n/n!}.$$
When $\mathfrak a_p=\mathcal I^p$ for all $p$ and for some fixed fractional ideal sheaf $\mathcal I$, denote $$\mult(\mathcal I)=_{\rm def}\mult(\mathfrak a_{\bullet}).$$
\end{defn}

\begin{rem}\label{alglocmult} When $\mathcal I$ is an ideal sheaf in $\mathcal O_X$, we understand the local cohomology group
$H^1_{\{x\}}(\mathcal I)$ as 
$\widetilde{\mathcal I}/\mathcal I$, 
where $\widetilde{\mathcal I}$ is obtained from $\mathcal I$ by removing the $\mathfrak m-$primary components in any primary ideal decomposition, where $\mathfrak m$ is the maximal ideal associated to the point $x$. Algebraically, $$\widetilde{\mathcal I}=(\mathcal I:\mathfrak m^{\infty})=_{\rm def}\bigcup_{p\geq 0}(\mathcal I:\mathfrak m^p).$$ By restricting to an affine neighborhood of $x$, $$H^1_{\{x\}}(\mathcal I)=H^0_{\{x\}}(\mathcal O_X/\mathcal I).$$\end{rem}

\noindent When $\mathfrak a_{\bullet}$ is a graded sequence of $\mathfrak m-$primary ideals in $\mathcal O_X$, then $\mult(\mathfrak a_{\bullet})$ coincides with the multiplicity defined in \cite[Sec.3.2]{LM}. If $\mathcal I$ is an $\mathfrak m-$primary ideal, then $\mult(\mathcal I)$ is the Hilbert-Samuel multiplicity of $\mathcal I$ at $\mathfrak m$.

\begin{rem} When $\pi:X'\to X$ is a projective birational morphism onto a normal quasi-projective variety $X$ of dimension at least two, with a distinguished point $x$ and $D$ is an integral Cartier divisor on $X'$ for which $\mathfrak a_p$ denotes the fractional coherent sheaf $\pi_*\mathcal O_{X'}(pD)$, then by definition $$\vol_x(D)=\mult(\mathfrak a_{\bullet}).$$\end{rem}
\begin{rem}\label{volbl} When $\mathfrak a$ is a fractional ideal sheaf on $X$ and $\mathcal O(1)$ denotes the relative Serre bundle on the blow-up of $X$ along $\mathfrak a$, then using \cite[Lem.5.4.24]{L}, $$\mult(\mathfrak a)=\vol_x(\mathcal O(1)).$$\end{rem}

The Fujita approximation result in \cite[Thm.3.8]{LM} states that for any graded sequence $\mathfrak a_{\bullet}$ of $\mathfrak m-$primary ideals, $$\mult(\mathfrak a_{\bullet})=\lim_{p\to\infty}\frac{\mult(\mathfrak a_p)}{p^n}.$$ Our contribution is noticing that the proof of the result extends to graded sequences of fractional ideals that are divisorial outside $x$ i.e. $\mathfrak a_p|_{X\setminus\{x\}}=\mathcal O_{X\setminus\{x\}}(pL)$ for some Cartier divisor $L$ on $X\setminus\{x\}$ and all $p$. In the language of local volumes, we have the following:

\begin{thm}\label{Fujita} Let $\pi:X'\to X$ be a projective birational morphism onto a normal quasi-projective variety $X$ of dimension $n\geq 2$ with a a fixed point $x$. Let $D=D_1+D_2$ be a sum of integral Cartier divisors on $X'$ such that $D_1$ is $\pi-$trivial outside $x$ and $D_2$ is effective and $\pi-$exceptional. Then $$\vol_x(D)=\lim_{p\to\infty}\frac{\mult(\pi_*\mathcal O_{X'}(pD))}{p^n}.$$
\end{thm}
\begin{paragraph}{Sketch of proof:}We can assume that $X$ is projective. The conditions on $D_1$ and $D_2$ imply that the fractional ideal sheaves $\mathfrak a_p=\pi_*\mathcal O_{X'}(pD)$ form a graded sequence that is divisorial outside $x$.\end{paragraph}
\par By \cite[Lem.3.9]{LM}, when $D$ lies over $x$, there exists an ample divisor $H$ on $X$ such that for every $p,k>0$, $$H^1(X,\mathcal O_X(pkH)\otimes(\pi_*\mathcal O_{X'}(pD))^k)=0$$ and the subspaces $H^0(X,\mathcal O_X(pH)\otimes\pi_*\mathcal O_{X'}(pD))\subseteq H^0(X,\mathcal O_X(pH))$ determine rational maps $\phi_p:X\to\mathbb PH^0(X,\mathcal O_X(pD))$ that are birational onto their image for all $p>0$. The $\mathfrak m-$primary hypothesis induced by $D$ lying over $x$ is only used to show that $\frac{\mathfrak a_1^{kp}}{\mathfrak a_p^k}$ is supported at $x$ for all $k,p>0$. But this conclusion also holds if the graded sequence $\mathfrak a_{\bullet}$ is divisorial outside $x$.
\par The proof of our result is then an almost verbatim copy of \cite[Thm.3.8]{LM}, noting that $\widetilde{\mathfrak a_p^k}=\widetilde{\mathfrak a_{pk}}$ for all $p,k>0$ and using the short exact sequences 
$$0\to \mathfrak a_p^k\to \widetilde{\mathfrak a_{pk}}\to H^1_{\{x\}}(\mathfrak a_p^k)\to 0.$$ Recall that $\widetilde{\mathfrak a}$ is obtained from $\mathfrak a$ by removing any $\mathfrak m-$primary components from any primary ideal decomposition. It is the same as $i_*i^*\mathfrak a$, where $i:X\setminus\{x\}\to X$ is the natural open embedding.

\begin{rem}Using Remark \ref{volbl} and Lemma \ref{invpull}, Theorem \ref{Fujita} implies that $\vol_x(D)$ is the limit of local volumes of $\mathbb Q-$Cartier, nef over $X$ divisors on blow-ups of $X'$, thus realizing the analogy with the global version of the Fujita approximation theorem.\end{rem} 

\begin{rem}\label{r2}The highly restrictive condition on $D$ in our Fujita approximation result is automatic when $\pi$ is an isomorphism outside $x$, which is the case for good resolutions of normal isolated singularities $\pi:(\widetilde X,E)\to(X,x)$. Even when $\pi$ is only a log-resolution of a normal isolated singularity, the divisor $K_{\widetilde X}+E$ satisfies the condition of Theorem \ref{Fujita} since $X\setminus\{x\}$ is non-singular.\end{rem}

For the rest of the subsection we look at $\mult(\mathcal I)$ for $\mathcal I$ an ideal sheaf on normal quasi-projective $X$ of dimension $n$ with a fixed point $x$ and compute two examples. This invariant has been studied in some cases in \cite{CHST}, where an example of an irrational local multiplicity is given, showing that, unlike the $\mathfrak m-$primary case, $\dim H^1_{\{x\}}(\mathcal I^p)$ is generally not asymptotically a polynomial function in $p$.

\begin{rem} Note that the proof of Theorem \ref{Fujita} carries for the graded family $\mathfrak a_p=\mathcal I^p$, therefore the $\limsup$ in the definition of $\mult(\mathcal I)$ can be replaced by $\lim$ and this limit is finite by Remark \ref{volbl} and Proposition \ref{volfin}. Compare \cite[Thm.1.3]{CHST}.\end{rem}

\begin{exap}[Monomial ideals] Let $I$ be a monomial ideal in $\mathbb C[X_1,\ldots,X_n]$ and let $\mathfrak m=(X_1,\ldots,X_n)$ be the irrelevant ideal corresponding to the origin of $\mathbb C^{n}$. Then $$(I^k:\mathfrak m^{\infty})=\bigcap_{i=1}^n(I^k:X_i^{\infty}).$$ 
For an arbitrary monomial ideal $J$, the ideal $(J:X_i^{\infty})$ can be computed as $\varphi_i^{-1}\varphi_i(J)$, where $$\varphi_i:\mathbb C[X_1,\ldots,X_n]\to\mathbb C[X_1,\ldots,\hat X_i,\ldots, X_n]$$ is the evaluation map determined by $\varphi_i(X_j)=X_j$ for $j\neq i$ and $\varphi_i(X_i)=1$. Geometrically, $J$ is determined by the set $A(J)$ of $n-$tuples of non-negative numbers $(a_1,\ldots,a_n)$ such that $X_1^{a_1}\cdot\ldots\cdot X_n^{a_n}$ belongs to $J$. Then $A(J:X_i^{\infty})$ is obtained by taking the integer coordinate points in the pre-image of the image of $A(J)$ via the projection onto the coordinate hyperplane that does not contain the $i-$th coordinate axis. Subsequently, $A(J:\mathfrak m^{\infty})=\bigcap_{i=1}^n A(J:X_i^{\infty})$ and $$\dim H^1_{\mathfrak m}(J)=\#(A(J:\mathfrak m^{\infty})\setminus A(J)).$$ Let $P(J)$ denote the convex span of $A(J)$ in $\mathbb R^{n}$ and let $\widetilde P(J)$ be the polyhedron obtained by intersecting the pre-images of the images of the projection of $P(J)$ onto all the coordinate hyperplanes. Then one checks that $$\widehat{h}^1_{\mathfrak m}(I)=n!\cdot\vol(\widetilde P(I)\setminus P(I)),$$ where the volume used in the right hand side is the euclidean one.
\par Figure \ref{fig:mon} shows how $\frac18(A(I^8:\mathfrak m^{\infty})\setminus A(I^8))$ approximates $\widetilde P(I)\setminus P(I)$ for the monomial ideal $(X^3,XY^3)\subset\mathbb C[X,Y]$. Note that in the two dimensional case, $(I:\mathfrak m^{\infty})$ is principal for any monomial ideal $I$, but this does not necessarily hold in higher dimension.
\begin{figure}[htb]
\begin{center}
\leavevmode
\includegraphics[width=0.4\textwidth]{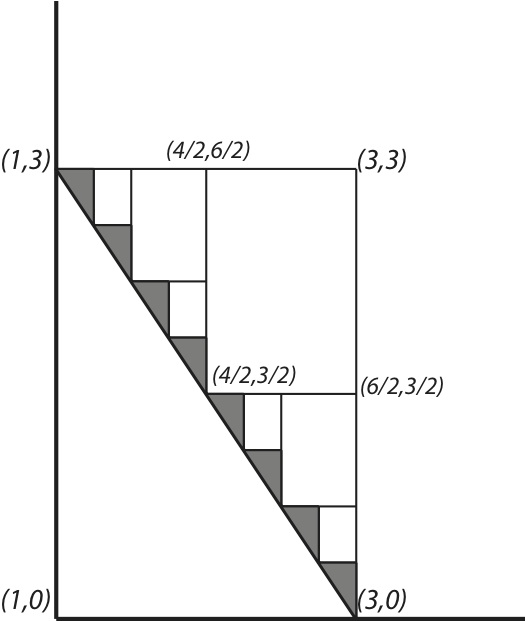}
\end{center}
\caption{$(X^3,XY^3)\subset\mathbb C[X,Y]$}
\label{fig:mon}
\end{figure}

\end{exap}

\begin{exap}[Toric ideals] Let $I$ be a monomial ideal inside the toric algebra $\mathbb C[S_{\sigma}]$ where $S_{\sigma}$ is the semigroup of integral points of a pointed (i.e contains no lines) rational convex cone $\sigma$ of dimension $n$. Let $\tau_1,\ldots,\tau_r$ be the minimal ray generators for $\sigma$. For a subset $V$ of $\sigma$ and $0\leq i\leq r$, let $V_i=\sigma\cap\bigcup_{k\geq 0}(V-k\cdot\tau_i)$. Geometrically, this is the trace left by $V$ inside $\sigma$ by sliding it in the direction of $-\tau_i$. 
\par If $P(I)$ is the convex hull of the set $A(I)$ defined as in the monomial case and if $\widetilde P(I)$ is the intersection of all $P(I)_i$, then $$\widehat{h}^1_{\mathfrak m}(I)=n!\cdot\vol(\widetilde P(I)\setminus P(I)).$$
\end{exap}

\subsection{Vanishing and log-convexity of local volumes}

Our first objective in this subsection is to study the vanishing of local volumes. We begin by recalling a few general facts about exceptional Cartier divisors. If $\pi:X'\to X$ is a projective birational morphism of quasi-projective varieties with $x$ a point on the normal variety $X$, the relative numerical space $N^1(X'/X)_{\mathbb R}$ contains two interesting subspaces. The first and largest of the two is the space of $\pi-$exceptional divisors that we denote ${\rm Exc}(\pi)$. That a relative numerical class is exceptional is well defined numerically is proved by the following result appearing as \cite[Lem.1.9]{BdFF}:

\begin{lem}\label{numdet}Let $\pi:X'\to X$ be a proper birational map with $X$ normal and let $\alpha\in N^1(X'/X)_{\mathbb R}$. Then there exists at most one exceptional $\mathbb R-$Cartier divisor $D$ on $X'$ whose relative numerical class over $X$ is $\alpha$. In particular, when $X'$ is normal and $\mathbb Q-$factorial, the numerical classes of the irreducible $\pi-$exceptional divisors form a basis of ${\rm Exc}(\pi)$.\end{lem}

\begin{prop}\label{Qfact} Assume that $X$ and $X'$ are both normal and $\mathbb Q-$factorial. Then $$N^1(X'/X)_{\mathbb R}={\rm Exc}(\pi).$$\end{prop}
\begin{proof}
We observe that any Cartier divisor $D$ on $X'$ is $\pi-$linearly equivalent, over $\mathbb Q$, to an exceptional divisor via $$D=\pi^*(\pi_*D)+(D-\pi^*\pi_*D).$$ The pullback by $\pi$ is well defined since the Weil divisor $\pi_*D$ is $\mathbb Q-$Cartier by assumption and $D-\pi^*\pi_*D$ is clearly exceptional.
\end{proof}

A subspace of ${\rm Exc}(\pi)$ we have seen to be relevant to the study of local volumes is formed by the divisors lying over $x$. We denote it by ${\rm Exc}_x(\pi)$. Studying the behavior of the local volume function on this space will prove important in connecting our work to the study of volumes for some $b-$divisors as developed in \cite{BdFF}. A particularly useful result, drawing on \cite[Lem.1-3-2]{KMM}, is \cite[Lem.4.5]{dFH}:

\begin{lem}\label{excvan} Let $\pi:X'\to X$ be a proper birational morphism from a non-singular variety $Y$ onto the normal variety $X$. Let $P$ and $N$ be effective divisors on $X'$ without common components and assume that $P$ is $\pi-$exceptional. Then $\pi_*\mathcal O_{X'}(P-N)=\pi_*\mathcal O_{X'}(-N).$\end{lem}

It is natural to ask what divisors in $N^1(X'/X)_{\mathbb R}$ have zero local volume over $x$. The answer to this question is well understood for volumes of Cartier divisors on projective varieties; we know that $\vol(D)>0$ is equivalent to $D$ being in the interior of the cone of pseudo-effective divisors (see \cite[Ch.2.2.C]{L}). In the local setting, we start by looking at the fiber over $x$.

\begin{prop}\label{evan} Let $\pi:X'\to X$ be a proper birational morphism of algebraic varieties with $X$ normal of dimension $n\geq 2$ and let $x\in X$. For $D\in {\rm Exc}_x(\pi)$, the vanishing $\vol_x(D)=0$ is equivalent to $D$ being effective.\end{prop}
\begin{proof}We can assume that $X$ is projective, that $\pi$ is a log-resolution and that $D$ is an integral divisor. If $D$ is effective, then $\pi_*\mathcal O_{X'}(mD)=\mathcal O_X$ for all $m\geq 0$ and so $\vol_x(D)=0$. Using Lemma \ref{excvan}, to complete the proof, it is enough to show that if $-D$ is effective, then $\vol_x(D)>0$. 
\par Let $\mathfrak m$ denote the maximal ideal sheaf on $X$ corresponding to $x$ and let $e(\mathcal I)$ denote the Hilbert-Samuel multiplicity at $x$ of an $\mathfrak m-$primary ideal sheaf $\mathcal I$. The idea is to show that there exists $r>0$ such that for all $m\geq 1$ we have an inclusion $$\pi_*\mathcal O_{X'}(mD)\subseteq\mathfrak m^{[m/r]},$$ because then $e(\pi_*\mathcal O_{X'}(mD))\geq e(m^{[m/r]})$, leading to $\vol_x(D)\geq e(\mathfrak m)/r^n>0$. This is a consequence of a result of Izumi (see \cite[Cor.3.5]{Iz}, or the presentation of Rees in \cite{R}).
\end{proof}

For abritrary Cartier divisors on $X'$ we can also give a precise answer, but one that does not provide satisfying geometric intuition.
\begin{prop}\label{ovan} Let $\pi:X'\to X$ be a proper birational morphism of algebraic varieties with $X$ normal. Assume $\dim X\geq 2$ and let $x$ be a point on $X$. If $D$ is a Cartier divisor on $X'$, then $\vol_x(D)=0$ if, and only if,  $\hl(m\widetilde D)=0$ for all $m\geq0$, where $\widetilde D$ is the pullback of $D$ to the normalization of $X'$.\end{prop}
\begin{proof}Since $\vol_x(D)=\vol_x(\widetilde D)$ and $\hl(m\widetilde D)$ is invariant under pullbacks from the normalization of $X'$ to another birational model of $X$, we can assume that $X'$ is non-singular and $\widetilde D=D$. One implication is clear. Since $\vol_x$ is $n-$homogeneous, we can assume without loss of generality that $\hl(D)\neq 0$. This means that $D$ is linearly equivalent to a divisor $F+G$ with $F$ effective (at least in a neighborhood of $E$) without components over $x$ and with $G$ a non-effective divisor lying over $x$. By Lemmas \ref{invpull}, \ref{l3} and by Proposition \ref{evan}, we then have $\vol_x(D)=\vol_x(F+G)\geq\vol_x(G)>0$.\end{proof}

\begin{rem}It is a consequence of Lemmas \ref{invpull}, \ref{l3} and \ref{excvan} that if $D$ is an exceptional divisor (not necessarily effective) without components lying over $x$ on the non-singular $X'$, then $\vol_x(D)=0$.\end{rem}

The conclusion of Proposition \ref{ovan} is not sufficient in understanding the vanishing of the local volume function on $N^1(X'/X)_{\mathbb R}$. We can prove the following partial result:
\begin{prop} Let $\pi:X'\to X$ be a proper birational map onto a normal algebraic variety $X$ with $x$ a distinguished point. Let $C_x$ denote the open cone in ${\rm Exc}_x(\pi)$ spanned by effective classes whose support is the entire divisorial component of the set theoretic fiber $\pi^{-1}\{x\}$. Then there exists an open cone $\mathcal C$ in $N^1(X'/X)_{\mathbb R}$ such that $\mathcal C\cap{\rm Exc}_x(\pi)=C_x$ and $\vol_x(D)=0$ for any $D\in\mathcal C$.\end{prop}
\begin{proof}We can assume that $X'$ is non-singular. Fix $E\in C_x$. We first show that for any Cartier divisor $D$ on $X'$ it holds that $\vol_x(D+tE)=0$ for $t\gg0$. By the monotonicity properties in Lemmas \ref{l2} and \ref{l3}, we can further assume $D$ is effective without components over $x$. With the notation in Lemma \ref{l3} and by the approximation result there, $$\vol_x(D+t\Delta_1)=\vol_x(D+t\Delta_1)-\vol_x(t\Delta_1)\leq\vol((t(\Delta_1+\Delta_2)+N)|_{\Delta_1})=0$$ for $t\gg0$ since $(-\Delta_1-\Delta_2)|_{\Delta_1}$ is ample and $\Delta_2|_{\Delta_1}$ is effective. There exists positive $r$ such that $rE>\Delta_1$. Then $\vol_x(D+trE)\leq\vol_x(D+t\Delta_1)$ by Lemma \ref{l2} and we conclude that $\vol_x(D+tE)=0$ for $t\gg0$.
\par Working as in the proof of Proposition \ref{cont}, the result follows. 
\end{proof}

\par We have seen in Theorem \ref{t1} that when $\pi:X'\to X$ is a projective birational morphism onto a normal quasi-projective variety $X$ of dimension $n$ at least two and when $x$ is a point on $X$, we have that $\vol_x$ is continuous and $n-$homogeneous function on $N^1(X'/X)_{\mathbb R}$. These properties are shared by volumes of Cartier divisors on projective varieties (see \cite[Ch.2.2.C]{L} or \cite{LM}). In the projective setting, it is known that the volume function is log-concave on the big cone (\cite[Cor.4.12]{LM}), meaning that $$\vol(\xi+\xi')^{1/n}\geq\vol(\xi)^{1/n}+\vol(\xi')^{1/n}$$ for any classes $\xi$ and $\xi'$ with non-zero volume. In our local setting, it is easy to construct examples of divisors $E-E'$ lying over $x$ such that $\vol_x(E-E')$ and $\vol_x(E'-E)$ are both non-zero and so we cannot expect log-concavity. Generalizing \cite[Rem.4.17]{BdFF} and \cite[Thm.4.15]{BdFF}, results developed in the setting of isolated singularities, we show that $\vol_x$ is log-convex when we restrict to divisors lying over $x$.

\begin{prop}\label{logconvex} Let $\pi:X'\to X$ be a projective birational morphism onto the normal quasi-projective variety $X$ of dimension $n\geq 2$ and let $x\in X$. Then $\vol_x:{\rm Exc}_x(\pi)\to\mathbb R_{\geq0}$ is log-convex.\end{prop}
\begin{proof} The idea is that by the Fujita approximation result in \cite[Thm.3.8]{LM}, when $D$ lies over $x$, we can understand $\vol_x(D)$ as an asymptotic Hilbert-Samuel multiplicity. Then we apply Teissier's inequality (\cite[Ex.1.6.9]{L}). Let $\mathfrak m$ denote the maximal ideal corresponding to $x\in X$ and for an $\mathfrak m-$primary ideal sheaf $\mathcal I$ on $X$, denote by $e(\mathcal I)$ its Hilbert-Samuel multiplicity.
\par By the continuity and homogeneity of $\vol_x$, we can reduce to working with integral Cartier divisors lying over $x$.  Let $D$ and $D'$ be two such and construct the graded families of $\mathfrak m-$primary ideals $\mathfrak a_m=\pi_*\mathcal O_{X'}(mD)$ and $\mathfrak a'_m=\pi_*\mathcal O_{X'}(mD')$.  By \cite[Thm.3.8]{LM}, $$\vol_x(D)=\lim_{m\to\infty}\frac{e(\mathfrak a_m)}{m^n}$$ and a similar equality holds for $\vol_x(D')$. Denoting $\mathfrak b_m=\pi_*\mathcal O_{X'}(m(D+D'))$, one has $$\mathfrak a_m\cdot\mathfrak a'_m\subseteq\mathfrak b_m,$$ therefore $e(\mathfrak b_m)\leq e(\mathfrak a_m\cdot\mathfrak a'_m)$. Teissier's inequality in \cite[Ex.1.6.9]{L} then implies $$e(\mathfrak b_m)^{1/n}\leq e(\mathfrak a_m\cdot\mathfrak a'_m)^{1/n}\leq e(\mathfrak a_m)^{1/n}+e(\mathfrak a'_m)^{1/n}$$ and the conclusion follows again by \cite[Thm.3.8]{LM}.\end{proof}

\begin{rem} Note that we did not restrict ourselves to working with classes having positive volume as was necessary in the projective setting.\end{rem}

When $\pi$ is an isomorphism outside $x$ and $X$ is $\mathbb Q-$factorial, Propositions \ref{logconvex} and \ref{Qfact} show that $\vol_x$ is log-convex on $N^1(X'/X)_{\mathbb R}$. We construct a toric example showing that this does not hold for general $\pi$.

\begin{exap}\label{tnc}Let $\sigma\subset\mathbb R^3$ be the cone spanned by the vectors $(0,1,0)$, $(0,0,1)$ and $(1,0,-2)$. Let $\Sigma$ be a refinement obtained by adding the rays spanned by $(1,1,1)$ and $(1,0,0)$ such that $X(\Sigma)$ is $\mathbb Q-$factorial. These determine a proper birational toric morphism $\pi:X(\Sigma)\to X(\sigma)$ that is not an isomorphism outside $x_{\sigma}$. Let $x=x_{\sigma}$ be the torus fixed point of $X(\sigma)$. On $X(\Sigma)$, let $D$ and $E$ be the torus invariant divisors associated to the rays $(1,0,-2)$ and $(1,1,1)$ respectively. We show that 
$$\vol_x(2D-\frac12 E)^{1/3}+\vol_x(2D-\frac 32 E)^{1/3}<\vol_x(4D-2E)^{1/3}=2\cdot\vol_x(2D-E)^{1/3}.$$
The idea is to study the function $\vol_x(2D-tE)$. By Example \ref{toric}, the volume $\vol_x(2D-tE)$ is computed as the normalized volume of the body $$B(t)=\{(x,y,z)\in\mathbb R^3\ :\  x\geq 0,\ y\geq0,\ z\geq 0,\ x-2z\geq-2,\ x+y+z\leq t\}.$$  
Let $S(t)$ be the simplex generated by $(0,0,0)$, $(t,0,0)$, $(0,t,0)$ and $(0,0,t)$. We have $B(t)=S(t)$ for $0\leq t\leq 1$ and $B(t)\subsetneq S(t)$ for $t>1$. Figure \ref{fig:toric} shows the polyhedron $B(3/2)$ corresponding to $2D-\frac32E$. The desired inequality follows easily from the linearity of $\vol(S(t))^{1/3}$.
\begin{figure}[htb]
\begin{center}
\leavevmode
\includegraphics[width=0.4\textwidth]{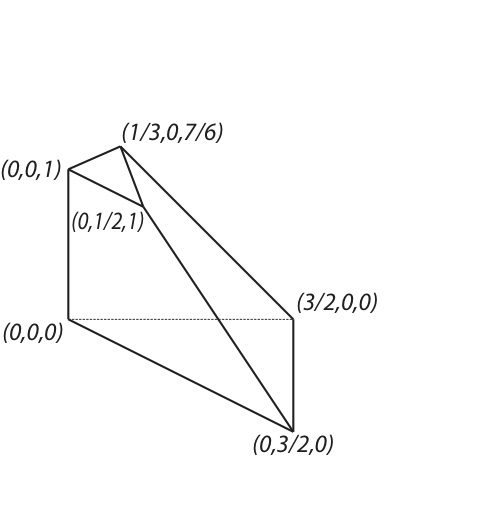}
\end{center}
\caption{$B(3/2)$}
\label{fig:toric}
\end{figure}

\end{exap}

\section{Plurigenera and volumes for normal isolated singularities}
In this section we introduce a notion of volume for normal isolated singularities of dimension at least two. This volume, that we will denote $\vol(X,x)$, is obtained in the first subsection as an asymptotic invariant associated to the growth rate of the plurigenera in the sense of Morales or Watanabe. We generalize to higher dimension several results of Wahl (\cite{JW1}) who introduced this volume on surfaces  and translate to our setting several results of Ishii (\cite{I2}). The second subsection studies the Kn\"oller plurigenera and the associated volume $\vol_{\gamma}(X,x)$ that, using results or Ishii (\cite{I2}) and of de Fernex and Hacon (\cite{dFH}), relates to the study of canonical singularities. We end with a series of examples in the third subsection. The results of this section are the motivation for our work and the foundation has been laid by the papers of Ishii (\cite{I2}) and Wahl (\cite{JW1}).  

\subsection{The Morales and the Watanabe plurigenera and $\vol(X,x)$}

The \textit{geometric genus} of a normal complex quasi-projective isolated singularity $(X,x)$ of dimension $n$ at least $2$, is defined as $$p_g(X,x)=_{\rm def}\dim_{\mathbb C}(R^{n-1}\pi_*\mathcal O_{\widetilde X})_x,$$ for $\pi:\widetilde X\to X$ an arbitrary resolution of singularities. 
Work of Yau in \cite{Y} shows that this invariant of the singularity can be computed analytically on $X$ as 
$$p_g(X,x)=\dim\frac{H^0(U\setminus\{x\},\mathcal O^{an}_X(K_X))}{L^2(U\setminus\{x\})},$$ where $U$ is a sufficiently small Stein neighborhood of $x$ in $X$ and $L^2(U\setminus\{x\})$ is the set of all square integrable forms on $U\setminus\{x\}$. Motivated by this alternate description, in \cite{W}, the \textit{plurigenera} of $(X,x)$ were introduced as
$$\delta_m(X,x)=_{\rm def}\dim\frac{H^0(U\setminus\{x\},\mathcal O^{an}_X(mK_X))}{L^{2/m}(U\setminus\{x\})},$$ with 
$L^{2/m}(U\setminus\{x\})$ now denoting the set of homomorphic $m-$forms $\omega$ on the sufficiently small $U\setminus\{x\}$ that satisfy $\int_{U\setminus\{x\}}(\omega\wedge\bar\omega)^{1/m}<\infty.$ The proofs of \cite[Thm.2.1]{S}, \cite[Thm.1.1]{S} and remarks in \cite{I2} provide an algebro-geometric approach to plurigenera at the expense of working again on resolutions. Let $\pi:\widetilde X\to X$ be a log-resolution of $(X,x)$ with $E$ the reduced fiber over $x$, let $U$ be an arbitrary affine neighborhood of $x$ and let $\widetilde U$ be the pre-image of $U$ in $\widetilde X$ via $\pi$. Then working in the algebraic category, $$\delta_m(X,x)=\dim\frac{H^0(\widetilde U\setminus E,\mathcal O_{\widetilde X}(mK_{\widetilde X}))}{H^0(\widetilde U,\mathcal O_{\widetilde X}(mK_{\widetilde X}+(m-1)E))}=\dim\frac{\mathcal O_X(mK_X)}{\pi_*\mathcal O_{\widetilde X}(mK_{\widetilde X}+(m-1)E)},$$ with the last equality holding, for choices of Weil canonical classes on $X$ and $\widetilde X$ such that $\pi_*K_{\widetilde X}=K_{\widetilde X}$ as Weil divisors, since $U$ is affine.

\begin{defn} Generalizing work in \cite{JW1} for the case of surfaces, the volume of the normal isolated singularity $(X,x)$ of dimension $n$ is defined as $$\vol(X,x)=_{\rm def}\limsup_{m\to\infty}\frac{\delta_m(X,x)}{m^n/n!}.$$\end{defn}
\noindent We would like to understand this volume as a local volume of some Cartier divisor on a log-resolution of $(X,x)$. For this, it turns out that a more convenient plurigenus is the one introduced by Morales in \cite{M}:
$$\lambda_m(X,x)=_{\rm def}\dim\frac{H^0(\widetilde U\setminus E,\mathcal O_{\widetilde X}(mK_{\widetilde X}))}{H^0(\widetilde U,\mathcal O_{\widetilde X}(m(K_{\widetilde X}+E)))},$$ for $\pi:\widetilde X\to X$ a log-resolution with $E$ the reduced fiber over $x$ and $\widetilde U$ the inverse image in $\widetilde X$ via $\pi$ of an affine neighborhood of $x$. By Remark \ref{r1}, $$\lambda_m(X,x)=\hl(m(K_{\widetilde X}+E)).$$ By \cite[Thm.5.2]{I2}, $$\vol(X,x)=\limsup_{m\to\infty}\frac{\lambda_m(X,x)}{m^n/n!}$$ and we can conclude that $$\vol(X,x)=\vol_x(K_{\widetilde X}+E),$$ independently of the chosen log-resolution.

\begin{rem} The classical literature usually requires that we work with good resolutions, i.e. that $\pi:\widetilde X\to X$ is a log-resolution that is an isomorphism outside $x$. To prove that the plurigenera are independent of the log-resolution, one applies the logarithmic ramification formula in \cite[Thm.11.5]{I1}, using that any two log-resolutions can be dominated by a third and that $X\setminus\{x\}$ is non-singular.\end{rem}

\begin{rem} If follows from Corollary \ref{truelim} that the $\limsup$ in the definition of $\vol(X,x)$ is an actual limit.\end{rem}

Generalizing a result for the volume of surface singularities (see \cite[Thm.2.8]{JW1}), we show that volumes of normal isolated singularities satisfy the following monotonicity property:

\begin{thm}\label{t2} Let $f:(X,x)\to (Y,y)$ be a finite morphism of normal isolated singularities i.e. $f$ is finite and set theoretically $f^{-1}\{y\}=\{x\}$. Then
$$\vol(X,x)\geq(\deg f)\cdot\vol(Y,y).$$
If $f$ is unramified away from $x$, then the previous inequality is an equality.
\end{thm}
\begin{proof} Let $\rho:(\widetilde Y,F)\to (Y,y)$ be a log resolution of $(Y,y)$. Let $Z$ be the normalization of $\widetilde Y$ in the fraction field of $X$ and let $u:(\widetilde X,E)\to (X,x)$ be a log-resolution factoring through a log-resolution of $Z$. We have a diagram:
$$\xymatrix{
\widetilde X\ar[dddr]_{\pi}\ar[drrr]^{\widetilde f}\ar[dr]^u\\
& Z\ar[dd]^{\tau}\ar[rr]^v & & \widetilde Y\ar[dd]^{\rho}\\
& & & \\
& X\ar[rr]^f & & Y}$$
We can assume that $\widetilde f$ has simple normal crossings for both the branching and ramification locus. We write the reduced branching locus as $F+R$, where $R$ has no components lying over $y$ and similarly write the reduced ramification locus as $E+S$ with $S$ having no components lying over $x$. 
\par A local study of forms with log poles at the generic points of each component of $E+S$ shows that $$K_{\widetilde X}+E+S=\widetilde f^*(K_{\widetilde Y}+F+R)+T,$$ where $T$ is an effective divisor that is exceptional for $\widetilde f$, hence also exceptional for $u$.
Note that $\widetilde f^*R-S$ is effective and write it as $P+Q$ with $P$ being supported on $S$ and with $Q$ being $u-$exceptional. Then
$$K_{\widetilde X}+E=\widetilde f^*(K_{\widetilde Y}+F)+P+(Q+T).$$
Since $P$ is supported on $S$, it has no components over $x$, so $$\vol(X,x)=\vol_x(K_{\widetilde X}+E)\geq\vol_x(\widetilde f^*(K_{\widetilde Y}+F)+(Q+T)),$$ by Lemma \ref{l3}. Since $Q+T$ is effective and $u-$exceptional and since $\vol_x$ is computed by pushing forward to $X$, $$\pi_*\mathcal O_{\widetilde X}(\widetilde f^*(K_{\widetilde Y}+F)+(Q+T))=\tau_*v^*\mathcal O_{\widetilde Y}(K_{\widetilde Y}+F),$$ and hence 
$$\vol_x(\widetilde f^*(K_{\widetilde Y}+F)+(Q+T))=\vol_x(v^*(K_{\widetilde Y}+F)).$$ By Proposition \ref{pfp}, $$\vol_x(v^*(K_{\widetilde Y}+F))=\deg(f)\cdot\vol_y(K_{\widetilde Y}+F)=\deg(f)\cdot\vol(Y,y).$$ 
\par When $f$ is unramified outside $x$, the divisors $R$, $S$ are zero and and with $T$ being again $u-$exceptional we obtain the required equality.
\end{proof}

An immediate consequence of this result is the following

\begin{cor}\label{endo}$\mbox{}$
\begin{enumerate}
\item If $f:(X,x)\to(Y,y)$ is a finite map of normal isolated singularities and $\vol(X,x)$ vanishes, then $\vol(Y,y)=0$.
\item If $(X,x)$ admits an endomorphism of degree at least two, then $\vol(X,x)=0$.
\end{enumerate}
\end{cor} 

In the surface case, \cite[Thm.2.8]{JW1} shows that $\vol(X,x)=0$ is equivalent to saying that $X$ has log-canonical singularities in the sense of \cite[Rem.2.4]{JW1}. In the $\mathbb Q-$Gorenstein case, this coincides with the usual definition of log-canonical. In higher dimension, as an immediate consequence of Proposition \ref{ovan}, or by \cite[Thm.4.2]{I2} it follows:

\begin{prop}\label{vn}Let $(X,x)$ be a normal complex quasi-projective normal isolated singularity of dimension $n$ at least two. Then $\vol(X,x)=0$ if, and only if, for all (any) log-resolutions $\pi:\widetilde X\to X$ with $E$ the reduced fiber over $x$, one has that 
$$\pi_*\mathcal O_{\widetilde X}(m(K_{\widetilde X}+E))=\mathcal O_X(mK_X),$$ for all non-negative $m$ and if, and only if, $\lambda_m(X,x)=0$ for all non-negative $m$.\end{prop}
\noindent In the previous result, we understand $\mathcal O_X(mK_X)$ as the sheaf of sections associated to a Weil canonical class $K_X$ chosen together with a canonical class on $\widetilde X$ such that $\pi_*K_{\widetilde X}=K_X$ as Weil divisors. 

\begin{rem} In the $\mathbb Q$-Gorenstein case, the conclusion of Proposition \ref{vn}, as in the case of surfaces, is the same as saying that $X$ is log-canonical. This result also appears in \cite{tw}. In general, following \cite{dFH}, we say $X$ is log-canonical if there exists an effective $\mathbb Q-$boundary $\Delta$ such that the pair $(X,\Delta)$ is log-canonical. With this definition, an inspection of \cite[Ex.4.20]{BdFF} and \cite[Ex.5.4]{BdFF} shows that there exist non $\mathbb Q-$Gorenstein isolated singularities $(X,x)$ that are not log-canonical, but $\vol(X,x)=0$.\end{rem}

Another result of Ishii (\cite[Thm.5.6]{I2}) that we translate to volumes studies hyperplane sections of normal isolated singularities.

\begin{prop} Let $(X,x)$ be an complex normal quasi-projective isolated singularity of dimension $n$ at least three. Let $(H,x)$ be a hyperplane section of $(X,x)$ that is again a normal isolated singularity. If $\vol(X,x)>0$, then $\vol(H,x)>0$.\end{prop}

\subsection{The Kn\"oller plurigenera} 

Another notion of plurigenera for a normal isolated singularity $(X,x)$, different from $\delta_m(X,x)$ and $\lambda_m(X,x)$, was introduced by Kn\"oller in \cite{Kn} and can be defined as
$$\gamma_m(X,x)=\dim\frac{\mathcal O_X(mK_X)}{\pi_*\mathcal O_{\widetilde X}(mK_{\widetilde X})}$$ for $\pi:\widetilde X\to X$ an arbitrary resolution of singularities. This is again an invariant of the singularity $(X,x)$, independent of the chosen resolution. The asymptotic behavior of $\gamma_m(X,x)=\hl(mK_{\widetilde X})$ is studied in \cite{I2}. Denoting
$$\vol_{\gamma}(X,x)=_{\rm def}\vol_x(K_{\widetilde X}),$$
the result in \cite[Thm.2.1]{I2}, or Proposition \ref{ovan} can be rephrased as:
\begin{prop} For a normal algebraic complex isolated singularity $(X,x)$ of dimension at least two, the following are equivalent:
\begin{enumerate}
\item $\vol_{\gamma}(X,x)=0$
\item $\gamma_m(X,x)=0$ for all non-negative $m$.
\end{enumerate}
\end{prop}

\noindent The following remark was kindly suggested by T. de Fernex.
\begin{rem}\label{can1}In \cite{dFH}, the authors generalize the notion of canonical singularities to normal varieties that are not necessarily $\mathbb Q-$Gorenstein and it is a consequence of \cite[Prop.8.2]{dFH} that a normal variety $X$ has canonical singularities if, and only if, for all sufficiently divisible $m\geq1$ and all (any) resolution $\pi:\widetilde X\to X$, it holds that $$\pi_*\mathcal O_{\widetilde X}(mK_{\widetilde X})=\mathcal O_X(mK_X),$$ with $K_X$ and $K_{\widetilde X}$ chosen such that $\pi_*K_{\widetilde X}=K_X$ as Weil classes.
\par When $(X,x)$ is an isolated singularity, since the $\limsup$ in the definition of $\vol_{\gamma}(X,x)$ is replaceable by $\lim$ for similar arguments as in the case of $\vol(X,x)$, the vanishing $\vol_{\gamma}(X,x)=0$ is equivalent to $(X,x)$ being canonical in the sense of \cite{dFH}.
\par Since in any case $\vol_{\gamma}(X,x)\geq\vol(X,x)$, we see that $\vol(X,x)=0$ for canonical singularities.
\end{rem}

We show that $\vol_{\gamma}$ does not exhibit the same monotonicity properties as $\vol(X,x)$ with respect to finite maps of normal isolated singularities by constructing a $\mathbb Q$-Gorenstein non-canonical isolated singularity carrying endomorphisms of arbitrarily high degree. 

\begin{exap} Let $(X,x)$ be the cone over $V=\mathbb P^{n-1}$ corresponding to the polarization $H=\mathcal O_{\mathbb P^{n-1}}(n+1)$. By Examples \ref{cone1} and \ref{cone2}, $$\vol_{\gamma}(X,x)=n\cdot\int_0^{\infty}\vol(K_V+H-tH)dt=n\cdot\int_0^{\infty}(1-t(n+1))^{n-1}dt=\frac 1{n+1}>0,$$ therefore $(X,x)$ is non-canonical and the other requirements are met.\end{exap}

However, we can prove the opposite to the inequality of Theorem \ref{t2} in the unramified case.

\begin{prop} Let $f:(X,x)\to(Y,y)$ be a finite morphism of complex normal isolated singularities of dimension $n$ at least two. Assume that $f$ is unramified away from $x$. Then $$\vol_{\gamma}(X,x)\leq(\deg f)\cdot\vol_{\gamma}(Y,y).$$\end{prop}
\begin{proof} Construct good resolutions $\pi:(\widetilde X,E)\to(X,x)$ and $\rho:(\widetilde Y,F)\to(Y,y)$ and a lift $\widetilde f:\widetilde X\to\widetilde Y$ for $f$. Then the ramification divisor $K_{\widetilde X}-\widetilde f^*K_{\widetilde Y}$ is effective. It is also exceptional for $\pi$ by assumption. We conclude by Proposition \ref{pfp} and Lemma \ref{l2}.
\end{proof}

\begin{cor} Under the assumptions of the previous proposition, if $(Y,y)$ has canonical singularities, then $(X,x)$ also has canonical singularities.\end{cor}
\begin{proof} The result is an immediate consequence of the proposition and Remark \ref{can1}.
\end{proof} 

\begin{rem} In this paper we refer to $\vol(X,x)$  and not to $\vol_{\gamma}(X,x)$ as the volume of the isolated singularity $(X,x)$.\end{rem}

\subsection{Examples}
We begin with a series of examples of normal isolated singularities $(X,x)$ where the volume is zero. We can usually show this by explicit computation of plurigenera or by noticing they carry non-invertible endomorphisms.

\begin{exap}[$\mathbb Q-$Gorenstein log-canonical case] Let $(X,x)$ be a $\mathbb Q-$Gorenstein log-canonical normal isolated singularity of dimension $n$. It is a consequence of Proposition \ref{vn} that $\vol(X,x)=0$, but we can also compute explicitly that $\lambda_m(X,x)=0$ for all non-negative, sufficiently divisible $m$. Pick $\pi:\widetilde X\to X$ a log-resolution with $E$ the reduced fiber over $x$. Since $\pi^*K_X$ is defined as a $\mathbb Q-$divisor, by Lemma \ref{invpull}, $$\lambda_m(X,x)=\hl(m(K_{\widetilde X}+E))=\hl(m(K_{\widetilde X}+E-\pi^*K_X))$$ for $m$ divisible enough so that $mK_X$ is Cartier. But $K_{\widetilde X}+E-\pi^*K_X$ is $\pi-$exceptional and effective by the log-canonical condition, so $\hl(m(K_{\widetilde X}+E-\pi^*K_X))=0$ for all sufficiently divisible $m$. By homogeneity, it follows that $\vol(X,x)=0$.\end{exap}

\begin{exap}[Finite quotient isolated singularities] Let $G$ be a finite group acting algebraically on a complex algebraic affine manifold $M$. Let $X=\Spec(\mathbb C[M]^G)$ be the quotient and assume it has a normal isolated singularity $x$. Then by Proposition \ref{pfp} and by the previous example, following ideas in Theorem \ref{t2}, we obtain $\vol(X,x)=0$.\end{exap}

\begin{exap}[Toric isolated singularities] We use the notation in Example \ref{toric}. Let $\sigma$ be an $n-$dimensional pointed rational cone. The condition that $(X(\sigma),x_{\sigma})$ be an isolated singularity is the same as saying that all the faces of non-maximal dimension of $\sigma$ are spanned as cones by a set of elements of $N$ that can be extended to a basis. Affine toric varieties carry Frobenius non-invertible endomorphisms and one checks that they are actually endomorphisms of the singularity $(X(\sigma),x_{\sigma})$ i.e. totally ramified at the isolated singularity, so $\vol(X(\sigma),x_{\sigma})=0$ by Corollary \ref{endo}.
\par It can be checked that for a toric resolution $\pi:(X(\Sigma),E)\to (X(\sigma),x_{\sigma})$, the divisor $K_{X(\Sigma)}+E$ is negative without components lying over $x_{\sigma}$ and then $\vol(X,x)=0$ by Lemma \ref{l3}.
\end{exap}

\begin{exap}[Cusp singularities]Tsuchihashi's cusp singularities provide yet another example of isolated singularities $(X,x)$ with $\vol(X,x)=0$.  See \cite[6.3]{BdFF} or \cite[Thm.1.16]{W} for explanations and \cite{T} for more on cusp singularities.\end{exap}

One of the simplest classes of isolated singularities that may have non-zero volume are quasi-homogeneous singularities.

\begin{exap}[Quasi-homogeneous singularities]\label{qh} We follow \cite[Def.1.10]{W}. Let $r_0,\ldots,r_n$ be positive rational numbers. Call a polynomial $f(x_0,\ldots,x_n)$ quasi-homogeneous of type $(r_0,\ldots,r_n),$ if it is a linear combination of monomials $x_0^{a_0}\cdot\ldots\cdot x_n^{a_n}$ with $\sum_{i=0}^na_ir_i=1$. When such a polynomial is sufficiently general, it's vanishing locus in $\mathbb C^{n+1}$ has an isolated singularity at the origin. We denote this singularity $(X(f),0)$ and $r(f)=r_0+\ldots+r_n$. By \cite[Exap.1.15]{W}, 
$$\vol(X(f),0)=\left\{\begin{array}{cc}0,&\mbox{if }r(f)\geq 1\\ \frac{(1-r(f))^n}{r_0\cdot\ldots\cdot r_n},&\mbox{if }r(f)\leq 1\end{array}\right..$$\end{exap}

\begin{exap}[Surface case]\label{sf2}By Example \ref{sf1}, the volume of a normal isolated surface singularity $(X,x)$ can be computed as $$\vol(X,x)=-P\cdot P,$$ where $K_{\widetilde X}+E=P+N$ is the relative Zariski decomposition on a good resolution $\pi:\widetilde X\to X$. In \cite[Prop.2.3]{JW1}, an algorithm for computing $P$ is described in terms of the combinatorial data of the dual graph of a good resolution.\end{exap}

Although the quasi-homogeneous and surface cases provide non-zero examples, they always provide rational values for the volume of the singularity. We will see that cone singularities provide irrational volumes already in dimension three.

\begin{exap}[Cone singularities]\label{cone2} If $(X,0)$ is a cone singularity constructed as $$\Spec\bigoplus_{m\geq 0}H^0(V,\mathcal O_V(mH))$$ for $(V,H)$ a polarized non-singular projective variety of dimension $n-1$, then by Example \ref{cone1}, using that $K_Y+E$ restricts to $K_V$ on $E$ by adjunction, $$\vol(X,0)=n\cdot\int_0^{\infty}\vol(K_V-tH)dt.$$ We see right away that $\vol(X,0)>0$ if, and only if, $V$ is of general type.
\end{exap}

In similar flavor to an example of Urbinati in \cite{U}, following a suggestion of Lazarsfeld, we show that there exist cone singularities yielding an irrational volume.
\begin{exap}[Irrational volume]\label{irr1} Choose two general integral classes $D$ and $L$ in the ample cone of $E\times E$, where $E$ is a general elliptic curve. Then, by the Lefschetz Theorem (\cite[Thm.6.8]{D}), $2D$ is globally generated and we can construct $V$, the cyclic double cover (see \cite[Prop.4.1.6]{L}) of $E\times E$ over a general section of $2D$. Let $g:V\to E\times E$ be the cover map. Note that $K_V=g^*D$. The volume of the cone singularity $(X,0)$ associated to $(V,g^*L)$ is then 
$$3\cdot\int_{0}^{\infty}\vol(g^*(D-tL))dt.$$ 
On abelian varieties, pseudo-effective and nef are equivalent notions for divisors and the volumes of such are computed as self-intersections. Let 
$$m=_{\rm def}\max\{t:\ D-tL\mbox{ is nef}\}.$$ It can be also characterized as the smallest solution to the equation $$(D-tL)^2=0.$$
One can compute, $$\vol(X,0)=\frac{4D^2L^2-4(DL)^2}{L^2}\cdot m+\frac{2(DL)D^2}{L^2}.$$ The study in \cite[Sec.1.5.B]{L} shows that the nef cone of $E\times E$ is a round quadratic cone for general $E$. Hence general choices for $D$ and $L$ produce a quadratic irrational $m$.
\end{exap} 

In \cite{JW1} it is proved that $\vol(X,x)$ is a topological invariant of the link of the surface singularity $(X,x)$. We give an example showing that this may fail already in dimension three. The idea for the construction comes from \cite[pag.36]{BdFF} and \cite[Ex.4.23]{BdFF} where, using the Ehresmann-Feldbau theorem, it is shown that if $f:(V,A)\to T$ is a smooth polarized family of non-singular projective varieties, then the links of the cone singularities associated to $(V_t,A_t)$ have the same diffemorphism type as $t$ varies in $T$. This is used to show that if $V$ is the family of blow-ups of $\mathbb P^2$ at ten or more points and if $(C_t,0_t)$ denotes the three dimensional cone singularity over $(V_t,A_t)$, for some appropriate polarization $A$, then the volume $\Vol(C_t,0_t)$ (that we discuss in the next section) is positive for very general $t$, but it does vanish for particular values of $t$. Since the $V_t$'s are all rational surfaces, $\vol(C_t,0_t)=0$ for any $t$, but we can construct an example where $\vol(C_t,0_t)$ is non-constant by passing to double covers of the family of blow-ups of $\mathbb P^2$ at three distinct points.

\begin{exap}\label{nottop}Let $g:S\to T$ be the smooth family of blow-ups of $\mathbb P^2$ at the three distinct points. There are line bundles $H$ and $E$ on $S$ such that for each $t\in T$, the divisor $H_t$ is the pullback of the hyperplane bundle via the blow-down to $\mathbb P^2$ and $E_t=E_{t,1}+E_{t,2}+E_{t,3}$ is the exceptional divisor of the blow-up. The geometry of $S_t$ differs according to whether $t$ consists of three collinear or non-collinear points, with the latter being the generic case. In both cases, $3H_t-E_t=-K_{S_t}$ is big and globally generated and $4H_t-E_t$ is ample and globally generated. It follows by \cite[Ex.1.8.23]{L} that $4(4H_t-E_t)$ is very ample. 
\par Let $t_0$ be a set of collinear points and choose a smooth divisor in the linear system $|4(4H_{t_0}-E_{t_0})|$ corresponding to a section $s_{t_0}$. By cohomology and base change (\cite[Thm.III.12.11]{H}), the section $s_{t_0}$ extends in a neighborhood of $t_0$ to a section $s$ of $4(4H-E)$. By further restricting $T$, we can assume that $s_t$ vanishes along a smooth divisor for all $t$ (see  \cite[Ex.III.10.2]{H}). Let $h:V\to S$ be the double cover corresponding to $s$. By \cite[Prop.4.1.6]{L}, the composition $f:V\to T$ is again a smooth family. We endow it with the fiberwise polarization given by $A=h^*(40H-3E)$. By results on \cite[pag.36]{BdFF}, the links of the cone singularities $(C_t,0_t)$ associated to $(V_t,A_t)$ are all diffeomorphic. We compute $\vol(C_t,0_t)$ and show that we get different answers when the tree points to be blown-up are collinear than when they are non-collinear. Note that $$K_t=_{\rm def}K_{V_t}=h_t^*(K_{S_t}+2(4H_t-E_t))=h_t^*(5H_t-E_t).$$
By Example \ref{cone2}, $$\vol(C_t,0_t)=3\cdot\int_0^{\infty}\vol(h^*(5H_t-E_t-s(40H_t-3E_t)))ds=6\cdot\int_0^{\infty}\vol((5-40s)H_t-(1-3s)E_t)ds.$$
We are reduced to working with volumes on $\mathbb P^2$ blown-up at three distinct points. For this we can use Zariski decompositions (see \cite[Thm.2.3.19, Cor.2.3.22]{L}) that can be explicitly computed for $aH_t+bE_t$ with $a,b\in\mathbb Z$ to show that $\vol(C_t,0_t)$ yields different values when $t$ corresponds to collinear points than when it corresponds to non-collinear points.

\end{exap}

\section{An alternative notion of volume due to Boucksom, de Fernex and Favre}

In this section we prove an inequality between our definition of volume for normal isolated singularities and one other volume, recently introduced in \cite{BdFF} also as a generalization of Wahl's work. We assume some familiarity with the notation of \cite{BdFF}.

\par Let $(X,x)$ be a complex normal quasi-projective isolated singularity of dimension $n$. A Weil canonical class $K_X$ on $X$ induces canonical classes $K_{X_{\pi}}$ on all resolutions $\pi:X_{\pi}\to X$ and we form the $\pi-$exceptional log-discrepancy divisor 
$$(\mathcal A_{\mathcal X/X})_{\pi}=_{\rm def}K_{X_{\pi}}+\Env(-K_X)_{\pi}+1_{X_{\pi}/X},$$ 
where $1_{X_{\pi}/X}$ is the reduced divisorial component of the full exceptional locus of $\pi$. The envelope of the canonical class, $\Env(-K_X)_{\pi}$, is computed as in \cite[Def.2.3]{BdFF}. Intuitively, $-\Env(-K_X)_{\pi}$ computes the pull-back $\pi^*K_X$. Pulling-back Weil divisors is a subtle problem, but we mention that $-\Env(-K_X)_{\pi}$ is indeed the pull-back of $K_X$ when $X$ is $\mathbb Q-$Gorenstein, which should serve as justification for calling $(\mathcal A_{\mathcal X/X})_{\pi}$ a log-discrepancy divisor. 
\par As $\pi$ varies through all possible resolutions $\pi:X_{\pi}\to X$, the divisors $K_{X_{\pi}}$, $\Env(-K_X)_{\pi}$, $1_{X_{\pi}/X}$ and $(\mathcal A_{\mathcal X/X})_{\pi}$ all glue to form $b-$divisors over $X$, i.e., each is an association of a divisor to every resolution $\pi$, an association that is compatible with push-forwards of Weil classes. $b-$divisors appear naturally in the study of divisorial valuations on the fraction field of $X$ and provide an efficient way of working simultaneously with all the resolutions of $X$. We denote by $\mathcal A_{\mathcal X/X}^0$ the component lying over $x$ of the $b-$divisor $\mathcal A_{\mathcal X/X}$. A consequence of the smoothness of $X\setminus\{x\}$ is: 

\begin{rem}\label{disceff} The $b-$divisor $\mathcal A_{\mathcal X/X}- \mathcal A_{\mathcal X/X}^0$ is effective and exceptional. \end{rem}

\begin{defn} Let $(X,x)$ be a complex normal quasi-projective isolated singularity of dimension $n$. The volume of $(X,x)$ in the sense of \cite{BdFF} is $$\Vol(X,x)=_{\rm def}-(\Env(\mathcal A_{\mathcal X/X}^0))^n.$$

The envelope of a $b-$divisor $D$, with a suitable boundedness condition,  is defined in \cite[Rem.2.15]{BdFF} as the piecewise infimum of the envelopes of $D_{\pi}$. The envelope of the divisor $D_{\pi}$ on $X_{\pi}$ is itself a $b-$divisor that is, in a suitable sense, the limit of the $\mathbb Q-$multiples $\mathcal O_m(1/m)$ of the relative Serre bundles $\mathcal O_m(1)$, which exist on the blow-ups of $X$ along the fractional ideal sheaves $\pi_*\mathcal O_{X_{\pi}}(mD)$. Intersections of nef $b-$divisors lying over $x$ are defined in \cite[Def.4.13]{BdFF}. They generalize the intersections $D_1\cdot\ldots\cdot D_n$ if the $D_i$ are all divisors on $X_{\pi}$ with support in the fiber over $x$.
\end{defn}

\begin{rem} The volume $\Vol(X,x)$ is also a generalization of Wahl's volume for isolated surface singularities by \cite[Prop.5.1]{BdFF}, therefore $\Vol(X,x)=\vol(X,x)$ in dimension two.\end{rem}

In arbitrary dimension, we prove the following inequality:

\begin{thm}\label{tcomp}Let $(X,x)$ be a complex normal quasi-projective isolated singularity of dimension $n$. Then $$\Vol(X,x)\geq\vol(X,x).$$\end{thm}
\begin{proof} By \cite[Rem.2.15]{BdFF}, for any resolution $\pi:X_{\pi}\to X$, we have $\Env(\mathcal A_{\mathcal X/X}^0)\leq\Env((\mathcal A_{\mathcal X/X}^0)_{\pi})$ and the monotonicity property of intersection numbers in \cite[Thm.4.14]{BdFF} shows $$\Vol(X,x)\geq -(\Env((\mathcal A_{\mathcal X/X}^0)_{\pi}))^n.$$
By \cite[Rem.4.17]{BdFF}, the latter is equal to $\vol_x((\mathcal A_{\mathcal X/X}^0)_{\pi})$, since $\vol_x$ and envelopes both are computed from pushforward sheaves. Remark \ref{disceff} and Lemma \ref{excvan} yield $$\vol_x((\mathcal A_{\mathcal X/X}^0)_{\pi})=\vol_x((\mathcal A_{\mathcal X/X})_{\pi})=\vol_x(K_{X_{\pi}}+\Env(-K_X)_{\pi}+E),$$ where now $\pi:(X_{\pi},E)\to(X,x)$ is a log-resolution. Since $\vol(X,x)=\vol_x(K_{X_{\pi}}+E)$, it suffices to prove that $$\vol_x((K_{X_{\pi}}+E)+\Env(-K_X)_{\pi})\geq \vol_x(K_{X_{\pi}}+E).$$ 
By \cite[Lem.2.9]{BdFF}, $\Env(-K_X)_{\pi}$ is $\pi-$movable, hence there exists a sequence of effective divisors $D_m$ on $X_{\pi}$ without components over $x$, a sequence that converges to $\Env(-K_X)_{\pi}$ in $N^1(X_{\pi}/X)$. We conclude by the continuity of $\vol_x$ and Lemma \ref{l3}.
\end{proof}

\begin{rem}When $X$ is $\mathbb Q-$Gorenstein, \cite[Prop.5.3]{BdFF} shows that $\Vol(X,x)=\vol(X,x)$.\end{rem} 
\noindent Aiming to extend this result to the numerically Gorenstein case (see \cite[Def.2.24]{BdFF}), we start with a lemma inspired by the proof or \cite[Prop.5.3]{BdFF} that allows us to compute $\Vol(X,x)$ on a fixed resolution in a particular case:

\begin{lem}\label{bd} Let $\pi:(X_{\pi},E)\to(X,x)$ be a log-resolution of a normal isolated singularity of dimension $n$ and assume $\Env(-K_X)_{\pi}$ is $\pi-$nef. Then $$\Vol(X,x)=\vol_x(K_{X_{\pi}}+\Env(-K_X)_{\pi}+E).$$\end{lem}
\begin{proof} \cite[Cor.2.12]{BdFF} proves that $\Env(-K_X)$ is Cartier, determined on $X_{\pi}$. Using \cite[Lem.3.2]{BdFF}, $$\mathcal A_{\mathcal X/X}-\overline{(\mathcal A_{\mathcal X/X})_{\pi}}$$ is effective and exceptional over $X$. The conclusion follows from Lemma \ref{excvan}, \cite[Cor.2.12]{BdFF}, Remark \ref{disceff} and \cite[Rem.4.17]{BdFF}.\end{proof}

\begin{prop}\label{numg} If $X$ is a numerically Gorenstein i.e. $\Env(K_X)+\Env(-K_X)=0$, then $$\Vol(X,x)=\vol(X,x).$$\end{prop}
\begin{proof}
The numerically Gorenstein condition implies that $\Env(\pm K_X)_{\pi}$ is $\pi-$numerically trivial on any non-singular model $X_{\pi}$.
We conclude using the numerical invariance of local volumes and Lemma \ref{bd}.
\end{proof}

\noindent As \cite[Thm.4.21]{BdFF} proves, the volume $\Vol(X,x)$ satisfies the same monotonicity property with respect to finite covers that $\vol(X,x)$ does:
\begin{rem} Let $f:(X,x)\to(Y,y)$ be a finite morphism of isolated singularities. Then $$\Vol(X,x)\geq(\deg f)\cdot\Vol(Y,y).$$\end{rem}

\begin{rem} Although $\vol(X,x)$ and $\Vol(X,x)$ are equal on surfaces and in the numerically Gorenstein case, they may differ in general. \cite[Exap.5.4]{BdFF} provides an example of a cone singularity where $\Vol(X,x)>\vol(X,x)=0$.\end{rem}

\noindent One advantage of $\vol(X,x)$ is that being determined on any log-resolution, it is usually easy to compute. On the other hand, since every resolution may bring new information, $\Vol(X,x)$ is usually hard to compute when it is non-zero. Lemma \ref{bd} provides examples when we can realize $\Vol(X,x)$ as a local volume $\vol_x$ on a fixed birational model. Applying this to cone singularities, we give an example of an irrational $\Vol(X,x)$.

\begin{lem}\label{cone3} Let $(V,H)$ be a polarized projective non-singular variety of dimension $n-1$, let $(X,0)$ be the associated cone singularity and let $\pi:(Y,E)\to(X,0)$ be the contraction of the zero section of $\Spec_{\mathcal O_V}\Sym^{\bullet}\mathcal O_V(H)$. Let $f:Y\to V$ be the vector bundle map. Then $$\Env(-K_X)_{\pi}=f^*(-K_V+M\cdot H),$$ with $M$ minimal such that $-K_V+M\cdot H$ is pseudo-effective.
\end{lem}
\begin{proof} Note that $\pi$ is a good resolution, hence $$\mathcal O_X(-mK_X)=\bigcup_{t\geq 0}\pi_*\mathcal O_Y(-mK_Y+tE)$$ and by coherence there exists minimal $t_m$ such that $$\mathcal O_X(-mK_X)=\pi_*\mathcal O_Y(-mK_Y+t_mE).$$ We get an induced inclusion that is actually an equality outside $E$: $$\mathcal O_X(-mK_X)\cdot\mathcal O_Y\to\mathcal O_Y(-mK_Y+t_mE).$$ Using the defining minimality property of $t_m$ and that $E$ is irreducible, one finds 
$$Z(-mK_X)_{\pi}=_{\rm def}(\mathcal O_X(-mK_X)\cdot\mathcal O_Y)^{\vee\vee}=\mathcal O_Y(-mK_Y+t_mE)$$
Observe that $X$ is affine, therefore the sheaves $\pi_*\mathcal O_Y(-mK_Y+tE)$ are determined by their global sections. But by the relations in Example \ref{cone1} and since $K_Y+E=f^*K_V$ by adjunction, 
$$H^0(Y,\mathcal O_Y(-mK_Y+tE))=\bigoplus_{k\geq 0}H^0(V,\mathcal O_V(-mK_V+(-t-m+k)H))$$ and it follows that $t_m$ is the maximal $t$ such that $\mathcal O_V(-mK_V+(-t-m)H)$ has sections. Recalling from \cite[Def.2.3]{BdFF} that $\Env(-K_X)=\lim_m(Z(-mK_X)/m)$ and setting $l=\lim_m(t_m/m)$, one finds that $$\Env(-K_X)_{\pi}=-K_Y+lE=f^*(-K_V-(l+1)H)$$ with $l$ maximal such that $-(K_V+(l+1)H)$ is pseudo-effective. Manifestly $M=-1-l$.
\end{proof}

\begin{cor}\label{cbd} With the same notation as before, assume that $\Env(-K_X)_{\pi}$ is also $\pi-$nef. Then 
$$\Vol(X,0)=\left\{\begin{array}{cc}M^n\cdot H^{n-1}& \mbox{, if }M\geq 0\\ 0&\mbox{, if } M<0\end{array}\right..$$\end{cor}
\begin{proof} Since the negative case follows similarly, we assume $M>0$. By Lemma \ref{bd}, Example \ref{cone1}, the preceding result and from the ampleness of $H$,  $$\Vol(X,0)=\vol_0(K_Y+E+f^*(-K_V+M\cdot H))=\vol_0(f^*(M\cdot H))=n\cdot \int_0^{\infty}\vol(M\cdot H-tH)dt=$$ $$=M^n\vol(H)=M^n\cdot H^{n-1}.$$
\end{proof}

\begin{exap} As in Example \ref{cone2}, with notation as in the preceding lemma, let $\mathcal E$ be a general elliptic curve. Let $D$ and $L$ be integral ample divisors on $\mathcal E\times \mathcal E$, let $g:V\to \mathcal E\times \mathcal E$ be the double cover over a general section of $\mathcal O_{\mathcal E\times\mathcal E}(2D)$ and denote $H=g^*L$. Note that $K_V=g^*D$. Then $\Env(-K_X)_{\pi}$ is $\pi-$nef because its restriction to $E$, the only positive dimensional fiber, is isomorphic to $-K_V+M\cdot H$ which is pseudo-effective; and on $V$, nef and pseudo-effective are equivalent notions for pull-backs of divisors from $\mathcal E\times\mathcal E$ e.g from \cite[Rem.4.1.7]{L}. By the previous corollary, noting that $M$ is positive since $-K_V+M\cdot H=g^*(-D+M\cdot L)$, $$\Vol(X,0)=M^3H^2.$$ We find that $\Vol(X,0)$ can be irrational by producing an example of $D$ and $L$ where $M^3$ is irrational. The same construction as in Example \ref{irr1} works.
\end{exap}

\section{Questions}

In this section we present some questions that our work has left open.

\subsection{The volume of a normal isolated singularity}

\begin{paragraph}{1. A topological question.} In \cite{JW1}, it is shown that the volume $\vol(X,x)$ of a normal surface singularity is a characteristic number, which means that it is a topological invariant of the link of the singularity and that it satisfies the monotonicity property of Theorem \ref{t2}. The monotonicity property holds in arbitrary dimension and we have seen that $\vol(X,x)$ is generally not a topological invariant of the link.\end{paragraph}

\begin{quest}Let $(X,x)$ be a Gorenstein (or only $\mathbb Q-$Gorenstein, or only numerically Gorenstein as in Proposition \ref{numg}) complex normal algebraic isolated singularity of dimension $n$. Is $\vol(X,x)$ a topological invariant of the link of the singularity?\end{quest}
\noindent The question is also open for $\Vol(X,x)$ as described in \cite{BdFF}.

\begin{paragraph}{2. Boundedness from below.} As illustrated by the cone example \ref{cone2}, one expects a similarity between the theory of volumes of normal isolated singularities and that of volumes of non-singular projective varieties. In the projective setting, it is shown in \cite{Tsuji}, \cite{Ta}, \cite{HM} and improved in \cite{HMX} that the volume of varieties of general type of dimension $n$ is bounded from below by a positive constant depending only on $n$. For cone singularities, C. Hacon observed that the formula in Example \ref{cone2} shows that replacing the polarization by arbitrarily high multiples produces a sequence of volumes tending to zero. We can ask\end{paragraph}

\begin{quest} Given a positive $n\geq 2$, does there exists a constant $C(n)>0$ such that $\vol(X,x)\geq C(n)$ for any complex normal isolated Gorenstein singularity $(X,x)$ of dimension $n$ with positive $\vol(X,x)$?\end{quest}

\noindent The previous question has a positive answer in the surface case by \cite{Ga}. In arbitrary dimension, it is also open for $\Vol(X,x)$ according to \cite{BdFF}. Proving boundedness from below for $\vol(X,x)$ would imply it as well for $\Vol(X,x)$ by Theorem \ref{tcomp}.

\begin{paragraph}{3. An irrational $\mathbb Q-$Gorenstein example.} We have constructed irrational examples for both $\vol(X,x)$ and $\Vol(X,x)$ by working on cone singularities. One can check that $\mathbb Q-$Gorenstein cone singularities have rational volume. Our only other familiar  examples of non-zero volumes come from quasi-homogeneous singularities (see Example \ref{qh}) and from surfaces, but we have seen that all these have rational volume as well.\end{paragraph}

\begin{quest}Does there exist a $\mathbb Q-$Gorenstein normal isolated singularity whose volume is irrational? Note that in this case $\vol_x(X,x)=\Vol(X,x)$.\end{quest}

\subsection{Local multiplicities}
For a (fractional) ideal sheaf $\mathcal I$ on a variety $X$ of dimension $n$ with a distinguished point $x$, we have defined the local multiplicity $$\mult(\mathcal I)=_{\rm def}\limsup_{p\to\infty}\frac{\dim H^1_{\{x\}}(\mathcal I^p)}{p^n/n!}$$ that coincides with the Hilbert-Samuel multiplicity when $\mathcal I$ is $\mathfrak m-$primary, where $\mathfrak m$ is the maximal ideal corresponding to $x$. In the said $\mathfrak m-$primary case, it is known that the Hilbert-Samuel multiplicity can be computed as the negative of the top self-intersection of the Serre line bundle on the blow-up of $X$ along $\mathcal I$. 
\par By work in \cite{CHST}, there exist (non $\mathfrak m-$primary) ideal sheaves $\mathcal I$ whose local multiplicity is irrational, therefore we cannot expect a simple intersection theoretic interpretation for $\mult(\mathcal I)$.

\begin{quest} Is there an asymptotic intersection theoretic interpretation for $\mult(\mathcal I)$ when $\mathcal I$ is not necessarily $\mathfrak m-$primary?\end{quest}

Via Remark \ref{volbl}, in the language of local volumes, we can also ask:

\begin{quest} Let $\pi:X'\to X$ be a projective birational morphism onto a normal $n-$dimensional quasi-projective variety $X$ and let $x$ be a point on $X$. Given $D$ a $\pi-$nef divisor on $X$, is there an asymptotic intersection theoretic description for $\vol_x(D)$?\end{quest}

\begin{rem}The difficulty in these questions lies in that the support of the Serre line bundle on the blow-up of $\mathcal I$ and that of $D$ are often non-proper. Following ideas in \cite{LM}, when $D$ is actually $\pi-$ample and not just $\pi-nef$, small and non-canonical steps can be taken to avoid these difficulties by writing local volumes as differences of volumes on projective completions.
\par Let $\pi:X'\to X$ be projective birational with $X$ normal and quasi-projective, and take $x$ a point on $X$. We assume that $X'$ is non-singular. Let $-D=-A-B$ be a $\pi-$ample divisor on $X'$ with $B$ effective lying over $x$ and $A$ effective without components over $x$.
\par By picking a completion of $X$, we can assume it is projective. As in \cite[Thm. 3.8]{LM}, there exists $H$ sufficiently ample on $X$ such that $\pi^*H-D$ is ample on $X'$ and $$\vol_x(-D)=\vol(\pi^*H-A)-\vol(\pi^*H-D).$$
Since $\pi^*H-D$ is ample, $\vol(\pi^*H-D)=(\pi^*H-D)^n$. Unfortunately, $\vol(\pi^*H-A)$ is not as accessible to intersection theoretic interpretations and the local nature of $\vol_x$ is not visible in this picture. 
\end{rem}

\begin{rem}With the notation in the previous question, when $(X,x)$ is a normal isolated singularity and $D$ is a $\pi-$nef Cartier divisor lying over $x$, by \cite[Rem.4.17]{BdFF}, $$\vol_x(D)=-D^n.$$\end{rem} 

\subsection{Local volumes}
\begin{paragraph}{1. Fujita approximation.} We proved Theorem \ref{Fujita} for a particular class of divisors, for which we were able to carry the arguments in \cite{LM}.\end{paragraph}
\begin{quest}Let $\pi:X'\to X$ be a projective birational morphism onto a normal quasi-projective variety $X$ of dimension $n\geq 2$ with a distinguished point $x$. Let $D$ be a Cartier divisor on $X'$. With notation as in \ref{SFuj}, is it true that $$\vol_x(D)=\lim_{p\to\infty}\frac{\mult(\pi_*\mathcal O_{X'}(pD))}{p^n}?$$\end{quest} 

\begin{paragraph}{2. Geometric interpretation.} The volume of a big divisor on a projective variety can be computed as an asymptotic moving self-intersection number (see \cite[Thm.11.4.11]{L}), generalizing that volumes of big and nef divisors are computed as top self-intersection numbers. In the local setting, we ask the following vague question:\end{paragraph}

\begin{quest} Does there exist a notion of "moving intersection numbers" that, asymptotically, gives a geometric interpretation to $\vol_x(D)$ when $D$ is a Cartier divisor on a birational modification of a normal quasi-projective variety $X$ with a distinguished point $x$?\end{quest}

\begin{rem}For divisors lying over $x$ on modifications $\pi:X'\to X$ of a normal isolated singularity $(X,x)$ of dimension $n$, we have by \cite[Thm.4.17]{BdFF} that $$\vol_x(D)=-(\Env(D))^n.$$ In this case, a good notion of "moving self-intersection number" for the divisor $D$ is, following the definition in \cite[Ex.1.3]{BdFF}, $$-(Z(\pi_*\mathcal O_{X'}(D)))^n.$$\end{rem}

\par As hinted by the previous remark, as well as the proof of \cite[Thm.11.4.11]{L}, it is expected that the "moving self-intersection number" of $D$ depends only on the Serre line bundle on the blow-up of $X$ along $\pi_*\mathcal O_{X'}(D)$. The difficulty is that the support of the latter is often non-proper.

\begin{paragraph}{3. The behavior of $\vol_x$ on $N^1(X'/X)_{\mathbb R}$.} When $\vol$ is the volume function on $N^1(X)_{\mathbb R}$ for $X$ a projective variety, we know by work in \cite[Ch.2]{L} that the locus where the volume does not vanish is the open convex cone of  big divisors.\end{paragraph} 

\begin{quest}Given a projective birational map $\pi:X'\to X$ onto normal quasi-projective $X$ with a distinguished point $x$, study the vanishing of $\vol_x$ on $N^1(X'/X)_{\mathbb R}$, or on the space of exceptional divisors ${\rm Exc}(\pi)$.\end{quest}

\noindent\textsc{Department of Mathematics, University of Michigan, Ann Arbor, MI 48109, USA
\\ Institute of Mathematics of the Romanian Academy, P. O. Box 1-764, RO-014700,
Bucharest, Romania}
\vskip.5cm
\textsc{E-mail:} mfulger@umich.edu

\end{document}